\documentclass[11pt]{article}%

\usepackage[margin=2cm]{geometry}
\usepackage[utf8]{inputenc}
\usepackage[T1]{fontenc}
\usepackage{lmodern}
\usepackage[french,english]{babel}
\usepackage{amsthm, amssymb, amsmath, amsfonts, mathrsfs, esint, textcomp, bm, xcolor}
\usepackage[colorlinks=true, pdfstartview=FitV, linkcolor=blue, citecolor=blue, urlcolor=blue,pagebackref=false]{hyperref}
\usepackage{mathtools}
\usepackage{tikz}
\usepackage{enumitem, imakeidx}
\usetikzlibrary{patterns}

\usepackage{microtype}

\definecolor{darkblue}{rgb}{0,0,0.7} 
\definecolor{darkred}{rgb}{0.9,0.1,0.1}
\definecolor{darkgreen}{rgb}{0,0.5,0}

\newtheorem{thm}{Theorem}[section]

\newtheorem{lem}[thm]{Lemma}

\theoremstyle{remark}
\newtheorem{rem}[thm]{Remark}
\theoremstyle{definition}

\usetikzlibrary{patterns}
\definecolor{asparagus}{rgb}{0.53, 0.66, 0.42}
\renewcommand{\leq}{\leqslant}
\renewcommand{\geq}{\geqslant}

\renewcommand{\subset}{\subseteq}

\newcommand{\E}{\mathbb{E}}

\newcommand{\F}{\mathcal{F}}

\newcommand{\N}{\mathbb{N}}

\newcommand{\1}{\mathbf{1}}
\newcommand{\R}{\mathbb{R}}

\newcommand{\Z}{\mathbb{Z}}

\renewcommand{\P}{\mathbb{P}}

\newcommand{\eps}{\varepsilon}
\renewcommand{\d}{{\mathrm{d}}}

\newcommand{\aeps}{\eps^\alpha}

\newcommand{\Rd}{\R^d}

\newcommand{\km}{k_{\textrm{max}}}
\newcommand{\I}{\mathcal{I}}

\DeclareMathOperator{\capacity}{Cap}

\mathtoolsset{showonlyrefs}

\newcommand{\avsum}{\mathop{\mathpalette\avsuminner\relax}\displaylimits}

\makeatletter
\newcommand\avsuminner[2]{%
  {\sbox0{$\m@th#1\sum$}%
   \vphantom{\usebox0}%
   \ooalign{%
     \hidewidth
     \smash{\vrule height\dimexpr\ht0+1pt\relax depth\dimexpr\dp0+1pt\relax}%
     \hidewidth\cr
     $\m@th#1\sum$\cr
   }%
  }%
}
\makeatother

\numberwithin{equation}{section}

\begin{document}

\begin{center}
{\Large Derivation of Darcy's law in randomly perforated domains}\\

\medskip

{A. GIUNTI}
\end{center}

\bigskip

\bigskip

\begin{center}
\begin{minipage}[c]{15cm}
{\small
{\bf Abstract:} We consider the homogenization of a Poisson problem or a Stokes system in a randomly punctured domain with Dirichlet boundary conditions. We assume that the holes are spherical and have random centres and radii. We impose that the average distance between the balls is of size $\eps$ and their average radius is $\eps^{\alpha}$, $\alpha \in (1; 3)$. We prove that, as in the periodic case \cite{Allaire_arma2}, the solutions converge to the solution of Darcy's law (or its scalar analogue in the case of Poisson). In the same spirit of \cite{GH, GHV}, we work under minimal conditions on the integrability of the random radii. These ensure that the problem is well-defined but do not rule out the onset of clusters of holes. }
\end{minipage}
\end{center}

\bigskip

We are interested in the effective behaviour of a Stokes system or a Poisson equation in a bounded domain $D^\eps \subset \R^3$, perforated by many random small holes $H^\eps$. We impose Dirichlet boundary conditions on the boundary of the holes and of the domain. Problems like the one studied in this paper arise mostly in fluid-dynamics where a Stokes system in a punctured domain models the flow of a viscous and incompressible fluid through many disjoint obstacles. We focus on the regime where the effective equation is given by Darcy's law or its scalar analogue in the case of the Poisson problem. For the latter, this corresponds to the case where the average density of harmonic capacity of the holes $H^\eps$ goes to infinity in the limit $\eps \downarrow 0$. In the case of Stokes the same is true, this time with the harmonic capacity being replaced by the so-called \textit{Stokes capacity}. This is a vectorial version of the harmonic capacity where the class of minimizers further satisfies the incompressibility constraint  (see \eqref{stokes.capacity}). 

\smallskip

We construct the randomly punctured domain $D^\eps$ as follows: Given $\alpha \in (1, 3)$ and a bounded $C^{1,1}$-domain $D\subset \R^3$, we define 
\begin{align}\label{holes}
D^\eps:= D \backslash H^\eps, \ \ \ \ \ \ H^\eps:= \bigcup_{z \in \Phi \cap \frac{1}{\eps}D} B_{\eps^\alpha \rho_{z}}(\eps z).
\end{align}
Here, the set of centres $\Phi$ is a Poisson point process of intensity $\lambda >0$ and the set $\frac 1 \eps D :=\{ x \in \R^3 \, \colon \, \eps x \in D \}$. The radii $\mathcal{R}= \{ \rho_z\}_{z\in \Phi} \subset [1; +\infty)$ are independent and identically distributed random variables satisfying for a constant $C< +\infty$
\begin{align}\label{integrability.p}
\E \bigl[ \rho^{\frac{3}{\alpha}} \bigr] \leq C.
\end{align}
This condition is minimal in order to ensure that, $\P$-almost surely and when $\eps$ is small, the set $H^\eps$ does not fully cover the domain $D$, hence implying that $D^\eps = \emptyset$  (see Lemma \ref{zero.one.law}). However, condition \eqref{integrability.p} does not prevent that, with high probability, the balls in $H^\eps$ do overlap.
 
\bigskip

For $\eps >0$ and $D^\eps$ as above, we consider the (weak) solution to either
\begin{align}\label{P.eps.p}
\begin{cases}
-\Delta u_\eps  = f \ \ \ &\text{in $D^\eps$}\\
u_\eps = 0 \ \ \ &\text{on $\partial D^\eps$}
\end{cases}
\end{align}
or to 
\begin{align}\label{P.eps.s}
\begin{cases}
-\Delta u_\eps + \nabla p_\eps = f \ \ \ &\text{in $D^\eps$}\\
\nabla \cdot u_\eps =0 \ \ \ &\text{in $D^\eps$}\\
u_\eps = 0 \ \ \ &\text{on $\partial D^\eps$}
\end{cases}
\end{align}
In the case of the Stokes system, we further assume that
\begin{align}\label{integrability.s}
\E \bigl[ \rho^{\frac{3}{\alpha} + \beta} \bigr] \leq C, \ \ \ \text{for some $\beta > 0$.}
\end{align}
We refer to the next section for a more detailed discussion on what conditions \eqref{integrability.p} and \eqref{integrability.s} entail in terms of the geometric properties of the set $H^\eps$.

\bigskip

It is easy to see that in the case of spherical periodic holes having distance $\eps$ and radius $\eps^\alpha$, $\alpha \in (1, 3]$, the density of harmonic capacity of $H^\eps$ is asymptotically of order $\eps^{-3+ \alpha}$; The same is true in the case of the Stokes capacity. When $\alpha=3$ these limits are thus finite. In the case of the Poisson problem,  the solutions to \eqref{P.eps.p} thus converge to the solution $u \in H^1_0(D)$ to $-\Delta u +\mu u = f$ in $D$, where the constant $\mu>0$ is the limit of the capacity density \cite{Cioranescu_Murat}. Similarly, the limit problem for \eqref{P.eps.s} is given by a \textit{Brinkmann system}, namely a Stokes system in $D$ with no-slip boundary conditions and with the additional term $\tilde \mu u$ in the system of equations \cite{Allaire_arma1}. The term  $\tilde \mu >0$ is as well strictly related to the limit of the Stokes capacity density. We also mention that, for holes that are periodic but not spherical, the term $\tilde \mu$ is a positive-definite matrix. For $\alpha \in (1; 3)$ as in the present paper, the solutions to \eqref{P.eps.p} or \eqref{P.eps.s} need to be rescaled by the factor $\eps^{-3+ \alpha}$ in order to converge to a non-trivial limit. The effective equations, in this case, are either $u = k f$ in $D$ or Darcy's law $u=K( f - \nabla p)$ in $D$ \cite{Allaire_arma2}. Here, $k, K$ are related to the rescaled limit of the density of capacity and admit a representation in terms of a corrector problem solved in the exterior domain $\R^3 \backslash B_1(0)$. 

\bigskip

When $\alpha=1$, namely when the distance between holes and their size have the same order $\eps$, the effective equations for \eqref{P.eps.p} and \eqref{P.eps.s} are as in the case $\alpha \in (1, 3)$; the effective constants $k, K$ obtained in the limit, however, are determined by a corrector problem of different nature. In this case indeed, there is only one microscopic scale $\eps$ and the relative distance between the connected components of the holes $H^\eps$ does not tends to infinity for $\eps \to 0$. This yields that the corrector equations are solved in the periodic cell and not in the exterior domain $\R^3\backslash B_1(0)$ \cite{Allaire_porous}.

\bigskip

For holes that are not periodic, the extremal regimes $\alpha\in \{1, 3\}$ have been rigorously studied both in deterministic and random settings. For $\alpha =3$ we mention, for instance \cite{Brillard1986-1987, Desvillettes2008, Hillairet2018, LEVY198311,MarchenkoKhruslov,papvar.tinyholes,Rubinstein1986,SanchezP82} and refer to the introductions in \cite{GH} and \cite{GHV} for a detailed overview of these results. We stress that the homogenization of \eqref{P.eps.p} and \eqref{P.eps.s} when $H^\eps$ is as in \eqref{holes} with $\alpha =3$ has been studied in the series of papers \cite{GH, GH_pressure, GHV}. These works prove the convergence to the effective equation under the minimal assumption that $H^\eps$ has finite averaged capacity density. There is no additional condition on the minimal distance between the balls in the set of $H^\eps$.

\bigskip

There are many works devoted also to the regime $\alpha=1$. We refer, in particular, to \cite{Beliaev_Koslov} where \eqref{P.eps.p} and \eqref{P.eps.s} are studied for a very general class of stationary and ergodic punctured domains. For these domains, the formulation of the corrector equation for the the effective quantities $k, K$ is solved in the probability space $(\Omega, \F , \P)$ generating the holes. 

\bigskip

There is fewer mathematical literature concerning the homogenization of \eqref{P.eps.p} or \eqref{P.eps.s} in the regime $\alpha \in (1; 3)$. For periodic holes, this has been studied in \cite{Allaire_arma2}. These results have been extended for certain regimes to compressible Navier-Stokes systems \cite{Richard_Sebastian} or to elliptic systems in the context of linear elasticity \cite{Jing_elasticity}. We are not aware of analogous results when the holes $H^\eps$ are not periodic. The present paper considers this problem when $H^\eps$ is random and, in the same spirit of \cite{GH, GHV}, allows that the balls in $H^\eps$ overlap and cluster. %Condition \eqref{integrability.p}, indeed, is the minimal assumption for holes generated as in \eqref{holes}.

\bigskip

The main result of this paper is the following: 
\begin{thm}\label{t.main}
Let $\sigma_\eps := \eps^{-\frac{3 - \alpha}{2}}$ and let $H^\eps$ and $D^\eps$ be the random sets defined in \eqref{holes}. 
\begin{itemize}
\item[(a)] Let $u_\eps \in H^1_0(D^\eps)$ solve \eqref{P.eps.p} with $f \in L^q(D)$ for $q \in (2; +\infty]$. Then, if the marked point process $(\Phi, \mathcal{R})$ satisfies \eqref{integrability.p}, for every $p \in [1; 2)$ we have that 
\begin{align}
\lim_{\eps \downarrow 0}\E \bigl[ \int_D | \sigma_\eps^2 u_\eps -  k f|^p \bigr]  = 0, \ \ \ \text{with $k:= (4\pi \lambda \E\bigl[ \rho \bigr])^{-1}$.}
\end{align}
Here, and in the rest of the paper, $\E \bigl[ \, \cdot \, \bigr]$ denotes the expectation under the probability measure for $(\Phi, \mathcal{R})$.

\medskip

\item[(b)] Let $u_\eps \in H^1_0(D^\eps; \R^3)$ solve \eqref{P.eps.s} with $f \in L^q(D; \R^3)$ for $q \in (2 ; +\infty]$. If  $(\Phi, \mathcal{R})$ satisfies \eqref{integrability.p}, then for every $p \in [1; 2)$ we have
\begin{align}
\lim_{\eps \downarrow 0}\E \bigl[ \int_D | \sigma_\eps^2 u_\eps - K (f - \nabla p^*)|^p \bigr] = 0, \ \ \ \text{with $K:= (6\pi \lambda \E\bigl[ \rho \bigr])^{-1}$}
\end{align}
and $p^* \in H^1(D)$ (weakly) solving
\begin{align}
\begin{cases}
-\nabla \cdot( \nabla p^* - f) = 0 \ \ \ &\text{in $D$}\\
(\nabla p^*-f) \cdot \nu = 0 \ \ \ \ &\text{on $\partial D$}
\end{cases}\ \ \ \ \ \ \fint_D p^*= 0.
\end{align}
\end{itemize}
\end{thm}

\smallskip

As mentioned above, condition \eqref{integrability.p} is minimal in order to ensure that the set $D^\eps$ is non-empty for $\P$-almost every realization.  A lower stochastic integrability assumption for the radii, indeed, yields that, in the limit $\eps \downarrow 0$ and $\P$-almost surely $H^\eps$, covers the full set $D$ (see Lemma \ref{zero.one.law} in the next section). By the Strong Law of the Large Numbers, condition \eqref{integrability.p} implies that the density of capacity is almost surely of order $\eps^{-3+\alpha}$ as in the periodic case.  As already remarked in \cite{GH} in the case $\alpha=3$, with \eqref{P.eps.s} we require that the radii satisfy the slightly stronger assumption \eqref{integrability.s}. While \eqref{integrability.p} seems to be the optimal condition in order to control the density of harmonic capacity, the lack of subadditivity of the Stokes capacity calls for a better control on the geometry of the set $H^\eps$. 

\bigskip

The ideas used in the proof of Theorem \ref{t.main} are an adaptation of the techniques used in \cite{Allaire_arma2, Cioranescu_Murat} for the periodic case. They are combined with the tools developed in \cite{GH, GHV} to tackle the case of domains having holes that may overlap. As shown in \cite{Allaire_arma2}, the uniform bounds on the sequences $\{ \sigma_\eps^2 u_\eps\}_{\eps>0}$, $\{ \sigma_\eps \nabla u_\eps\}_{\eps>0}$ are obtained by means of a Poincar\'e's inequality for functions that vanish on $\partial D^\eps$. If $v \in H^1(D^\eps)$, since the function vanishes on the holes $H^\eps$, the constant in the Poincar\'e' s inequality is of order $\sigma_\eps^{-1}<< 1$. If $v \in H^1_0(D)$, this would instead be of order $1$ (dependent on the domain $D$). Note that, as for $\alpha=3$ we have $\sigma_\eps = 1$, there is no gain in using a Poincar'e's inequality in $H^1_0(D^\eps)$ instead of in $H^1_0(D)$ in this regime. In the case of centres of $H^\eps$ that are distributed like a Poisson point process, the is a low probability that some regions of $D^\eps$ have few holes, thus leading to a worse Poincar\'e's constant. This causes the lack of uniform bounds for the family $\{\sigma_\eps^2 u_\eps\}_{\eps>0}$ in $L^2(D)$.  

\bigskip

Equipped with uniform bounds for the rescaled solutions of \eqref{P.eps.p}, one may prove Theorem \ref{t.main}, $(a)$ by constructing suitable oscillating test functions $\{w_\eps \}_{\eps >0}$. These allow to pass to the limit in the equation and identify the effective problem. We stress that a crucial ingredient in these arguments is given by the quantitative bounds obtained in \cite{G} in the case $\alpha =3$. These bounds may indeed also be extended to the current setting sot that the rate of convergence of the measures $-\sigma_\eps^{-2}\Delta w_\eps \in H^{-1}(D)$ is quantified. This allows to control the convergence of the duality term $\langle -\Delta w_\eps ; u_\eps \rangle_{H^{-1}(D) ; H^1_0(D)}$. There is a fine balance the convergence of $-\sigma_\eps^2\Delta w_\eps$ with the right space where we have uniform bounds for $\{ \sigma_\eps^2 u_\eps \}_{\eps >0}$. In contrast with the periodic case, the  unboundedness of $\{ \sigma_\eps^2 u_\eps \}_{\eps >0}$ in $L^2(D)$ requires for a careful study of the duality term above. For the precise statements, we refer to \eqref{conv.Delta} in Lemma \ref{l.oscillating.p} and Lemma \ref{conv.measure}. The same ideas sketched here apply also to the case of solutions to \eqref{P.eps.s}. This time, the oscillating test functions $\{ w_\eps \}_{\eps >0}$ are replaced by the reduction operator $R_\eps$ of Lemma \ref{l.reduction.2}.

\bigskip

\begin{rem}\label{rem.variations} We comment below on some variations and corollaries of Theorem \ref{t.main}:
\begin{itemize}

\item[$(i)$] If $\Phi = \Z^d$ or is a stationary point process satisfying for a finite constant $C < +\infty$
$$
\max_{z_i, z_j \in \Phi} |z_i - z_j |  < C \ \ \ \text{ $\P$-almost surely,}
$$
then the convergence of Theorem \ref{t.main} holds also with $p=2$. In this case, indeed, we may drop the logarithmic factor in the bounds of Lemma \ref{l.unif.bounds}. 

The assumption $\mathcal{R}\subset [1; +\infty)$ may be also weakened to $\mathcal{R}\subset [0; +\infty)$, provided that
$$
\E \bigl[\rho^{-\gamma}\bigr] < +\infty,
$$
for an exponent $\gamma \in (1; +\infty]$. In this case, the convergence of Theorem \ref{t.main} holds in $L^p(D)$ for $p \in [1;  \bar p)$ with $\bar p= \bar p(\gamma) \in [1; 2)$ such that $\bar p(\gamma) \to 2$ when $\gamma \to +\infty$. 

\smallskip

\item[$(ii)$] A careful inspection of the proof of Theorem \ref{t.main} yields that, under assumption \eqref{integrability.s} and for a source $f\in W^{1,\infty}$, the convergences in both $(a)$ and $(b)$ may be upgraded to 
\begin{align}
\E \bigl[ \int_D | \sigma_\eps^2 u_\eps - u|^p \bigr]  \lesssim \eps^\kappa,
\end{align}
for an exponent $\kappa >0$ depending on $\alpha, \beta$.

\smallskip

\item[$(iii)$] The quenched version of Theorem \ref{t.main}, namely the $\P$-almost sure convergence of the families in $L^p(D)$, holds as well provided that we restrict to any vanishing sequence $\{ \eps_j\}_{j\in \N}$ that converges fast enough. For instance, it suffices that $j^{\frac 1 3 +\epsilon} \eps_j \to 0$, $\epsilon >0$. It is a technical but easy argument to observe that, under this assumption, limits \eqref{conv.Delta} of Lemma \ref{l.oscillating.p}  and \eqref{aver.R}-\eqref{meas.R}  of Lemma \ref{l.reduction.2} vanish also $\P$-almost surely.  From these, the quenched version of Theorem \ref{t.main} may be shown as done in the annealed case. To control the limits in \eqref{conv.Delta}, \eqref{aver.R} and \eqref{meas.R} without taking the expectation, one may follow the same lines of the current proof and control most of the terms by the Strong Law of Large Numbers. Condition $j^{\frac 1 3 + \epsilon} \eps_j \to 0$ on the speed of the convergence for $\{\eps_j\}_{j\in\N}$ is needed in order to obtain quenched bounds for the term in \eqref{clt} by means of Borel-Cantelli's Lemma. 

\smallskip

\item[$(iv)$] The analogue of Theorem \ref{t.main} holds also for a general dimension $d \geq 3$ if we consider the values $\alpha \in (1; \frac{d}{d-2})$ and rescale the solutions by $\sigma_\eps^2= \eps^{-\frac{d}{d-2} + \alpha}$. In this case, \eqref{integrability.p} and \eqref{integrability.s} hold with the exponent $\frac{3}{\alpha}$ replaced by $\frac{d}{\alpha}$.

\end{itemize}
\end{rem}

\bigskip

The paper is structured as follows: In the next section we describe the setting and introduce the notation that we use throughout the proofs. Subsection \ref{sub.integrability} is devoted to discussing the minimality of assumption \eqref{integrability.p} and what condition \eqref{integrability.s} implies on the geometry of the holes $H^\eps$. In Section \ref{s.uniform}, we show the uniform bounds on the family $\{ \sigma_\eps^2 u_\eps \}_{\eps >0}$, with $u_\eps$ solving \eqref{P.eps.p} or \eqref{P.eps.s}. In Section \ref{s.thm.a} we argue Theorem \ref{t.main} in case $(a)$, while in Section \ref{s.thm.b} we adapt it to case $(b)$. The proof of case $(b)$ is conceptually similar to the one for $(a)$, but it is technically more challenging. It heavily relies on the geometric properties of the holes implied by condition \eqref{integrability.s}.  Finally, Section \ref{s.appendix} contains the proof of the main auxiliary results used throughout the paper.  

\section{Setting and notation}\label{s.process}
Let $D \subset \R^3$ be an open set having $C^{1,1}$-boundary. We assume that $D$ is star-shaped with respect to a point $x_0 \in \R^3$.  This assumption is purely technical and allows us to give an easier formulation for the set of holes $H^\eps$. With no loss of generality we assume that $x_0=0$.

\bigskip

The process $(\Phi ; \mathcal{R})$ is a stationary marked point process on $\R^3$ having identically and independent distributed marks on $[1 +\infty)$. In other words, $(\Phi ; \mathcal{R})$ may be seen as a Poisson point process on the space $\R^3 \times [1;+\infty)$, having intensity $\tilde \lambda(x, \rho)= \lambda f(\rho)$.  The expectation in \eqref{integrability.p} or \eqref{integrability.s} is therefore taken with respect to the measure $f(\rho) \d \rho$. We denote by $(\Omega; \mathcal{F}, \mathbb{P})$ the probability space associated to $(\Phi, \mathcal{R})$, so that the random sets in \eqref{holes} and the random fields solving \eqref{P.eps.p} or \eqref{P.eps.s} may be written as $H^\eps= H^\eps(\omega)$, $D^\eps=D^\eps(\omega)$ and $u_\eps(\omega; \cdot)$, respectively. The set of realizations $\Omega$ may be seen as the set of atomic measures $\sum_{n \in \N} \delta_{(z_n, \rho_n)}$ in $\R^3 \times [1; +\infty)$ or, equivalently, as the set of (unordered) collections $\{ (z_n , \rho_n) \}_{ n\in \N} \subset \R^3 \times [1; +\infty)$. 

\bigskip

We choose as $\F$ the smallest $\sigma$-algebra such that the random variables $N(B): \Omega \to \N$, $\omega \mapsto  \#\{ \omega \cap B \}$ are measurable for every set $B \subset \R^4$ the Borel $\sigma$-algebra $\mathcal{B}_{\R^{4}}$. Here and throughout the paper, $\#$ stands for the cardinality of the set considered. For every $p \in [1; +\infty)$ we define the space $L^p(\Omega)$ as the space of ($\F$-measurable) random variables $F: \Omega \to \R$ endowed with the norm $\E\bigl[ |F(\omega)|^p \bigr]^{\frac 1 p}$. For $p=+\infty$, we set $L^\infty(\Omega)$ as the space of $\P$-essentially bounded random variables. We denote by $L^p(\Omega \times D)$, $p \in [1; +\infty)$, the space of random fields $F: \Omega \times \R^3 \to \R$ that are measurable with respect to the product $\sigma$-algebra and such that $\E \bigl[ \int_D |F(\omega, x)|^p \d x \bigr]^{\frac 1 p} < +\infty$. The spaces $L^p(\Omega), L^p(\Omega \times \R^3)$ are separable for $p \in [1, +\infty)$ and reflexive for $p \in (1, +\infty)$ (see e.g. \cite{Bruckner_Thomson}[Section 13,4]). The same definition, with obvious modifications, holds in the case of the target space $\R$ replaced by $\R^3$.

\bigskip

We often appeal to the Strong Law of Large Numbers (SSLN) for averaged sums of the form
$$
\#(\Phi \cap B_R)^{-1} \sum_{z\in \Phi \cap B_R} X_z,
$$
where $\{X_z \}_{z\in \Phi^\eps(D)}$ are identically distributed random variables that have sufficiently decaying correlations. Here, we send the radius of the ball $B_R$ to infinity. It is well-known that such results hold and we refer to \cite{GHV}[Section 5] for a detailed proof of the result that is tailored to the current setting.

\subsection{Notation}
We use the notation $\lesssim$ or $\gtrsim$ for $\leq C$ or $\geq C$ where the constant depends only on $\alpha$, $\lambda$, $D$ and, in case $(b)$, also on $\beta$ in \eqref{integrability.s}. Given a parameter $p \in \R$, we use the notation $\lesssim_p$ if the implicit constant also depends on the value $p$. For $r>0$, we write $B_r$ for the ball of radius $r$ centred in the origin of $\R^3$. We denote by $\langle \, \cdot \, ; \, \cdot \, \rangle$ the duality bracket between the spaces $H^{-1}(D)$ and $H^1_0(D)$.

\smallskip

When no ambiguity occurs, we skip the argument $\omega \in \Omega$ in all the random objects considered in the paper. If $(\Phi ; \mathcal{R})$ is as in the previous subsection, for a set $A \subset \R^d$, we define
 \begin{align}\label{psi.eps}
\Phi^\eps(A):= \bigl\{ z \in \Phi \, \colon \, \eps z \in A  \bigr\}, \ \ \ N^\eps(A):= \#\Phi^\eps(A).
 \end{align}
 For $x \in \R^3$, we define the random variables
 \begin{align}\label{distance}
 d_x:= \frac 1 2 \min_{z \in \Phi \atop z \neq x} |z- x|, \ \ \ R_x:= \min \bigl\{ d_x, \frac 1 2 \bigr\},\ \ \ d_{x,\eps}:= \eps d_x, \ \ \ R_{\eps,x}:= \eps R_x.
 \end{align}

\subsection{On the assumptions on the radii}\label{sub.integrability}
In this subsection we discuss the choice of assumptions \eqref{integrability.p} and \eqref{integrability.s} in Theorem \ref{t.main}.  We postpone to the Appendix the proofs of the statements. The next result states that assumption \eqref{integrability.p} is sufficient to have only microscopic holes whose size vanishes in the limit $\eps \downarrow 0$. Moreover, it is also necessary in order to have that holes $H^\eps$ do not cover the full domain $D$.
{
\begin{lem}\label{zero.one.law}
The following conditions are equivalent:
\begin{itemize}
\item[(i)] The process satisfies \eqref{integrability.p};
\item[(ii)] For $\P$-almost every realization and for every $\eps$ small enough the set $D^\eps \neq \emptyset$.
\end{itemize}
Furthermore, $(i)$( or $(ii)$) implies that for $\P$-almost realization $\lim_{\eps \downarrow 0}|D^\eps| =|D|$.
\end{lem}}

\smallskip

In the following result we provide the geometric information on $H^\eps$ that may be inferred by strengthening condition \eqref{integrability.p} to \eqref{integrability.s}. Roughly speaking, the next lemma tells that, under condition \eqref{integrability.s}, we have a control on the maximum number of holes \textit{of comparable size} that intersect. More precisely, we may discretize the range of the size of the radii $\{ \rho_z\}_{z\in \Phi^\eps(D)}$ and partition the set of centres $\Phi^\eps(D)$ according to the order of magnitude of the associated radii. The next statement says that there exists an $M\in \N$ (that is independent from the realization $\omega \in \Omega$) such that, provided that the step-size of the previous discretization is small enough, each sub-collection contains at most $M$ holes that overlap when dilated by a factor $4$. This result allows to treat also the case of the Stokes system in Theorem \ref{t.main}, (b) and motivates the need of the stronger assumption \eqref{integrability.s} in that setting.

\begin{lem}\label{l.borel.cantelli}
Let $(\Phi, \mathcal{\R})$ satisfy \eqref{integrability.s}. Then:
\begin{itemize}

\item[(i)] There exists $\kappa = \kappa(\alpha, \beta) > 0$, $\km= \km(\alpha, \beta), M=M(\alpha, \beta) \in \N$ such that for $\P$-almost every realization and for every $\eps$ small enough it holds
\begin{align}\label{max.radii}
\sup_{z \in \Phi^\eps(D)} \aeps\rho_z \leq \eps^{\kappa}
\end{align}
and we may rewrite  
\begin{equation}\label{partition.magnitude}
H_\eps= \bigcup_{i=1}^{\km} \bigcup_{ z \in I_{i,\eps}} B_{\aeps\rho_z}(\eps z),  \ \ \ \ \inf_{z \in I_{\eps,i}} \aeps \rho_z \geq \eps^\kappa  \sup_{z \in I_{\eps,i-2}} \aeps \rho_z \ \ \ \text{for $i=1, \cdots, \km$}
\end{equation}
such that for every $i=1, \cdots \km$
\begin{align}\label{no.overlapping.borel}
\{ B_{4\aeps\rho_z}(\eps z)\}_{z \in I_{i, \eps} \cup I_{i-1,\eps}}, \ \ \text{contains at most $M$ elements that intersect.}
\end{align}

\smallskip

\item[(ii)] For every $\delta > 0$ there exists $\eps_0=\eps_0(\delta)> 0$ and a set $B \in \mathcal{F}$ such that $\P(B) \geq 1-\delta$ and for every $\omega \in B$ and $\eps \leq \eps_0$ inequality \eqref{max.radii} holds and there exists a partition of $H^\eps$ satisfying \eqref{partition.magnitude}-\eqref{no.overlapping.borel}.

\end{itemize}
\end{lem}

\section{Uniform bounds}\label{s.uniform}
In this section we provide uniform bounds for the family $\{ \sigma_\eps^2 u_\eps \}_{\eps >0}$ and $\{ \sigma_\eps \nabla u_\eps \}_{\eps >0}$. We stress that, as in \cite{Allaire_arma2}, this is done by relying on a Poincar\'e's inequality for functions that vanish in the holes $H^\eps$. The order of magnitude of the typical size (i.e. $\eps^\alpha$) and distance (i.e. $\eps$) of the holes yields that the Poincar\'e's constant scales as the factor $\sigma_\eps$ introduced in Theorem \ref{t.main}. This, combined with the energy estimate for \eqref{P.eps.p} or \eqref{P.eps.s}, allows to obtain the bounds on the rescaled solutions.  We mention that the next results contain both annealed and quenched uniform bounds. The quenched versions are not needed to prove Theorem \ref{t.main}, but may be used to prove the quenched analogue described in Remark \ref{rem.variations}, (iii).

\begin{lem}\label{l.unif.bounds}
Let $u_\eps$ be is as in Theorem \ref{t.main}. Then for every $p \in [1; 2)$
\begin{align}\label{energy}
&\limsup_{\eps \downarrow 0}\E \bigl[ \int_D |\sigma_\eps \nabla u_\eps|^2 + |\log\eps|^{-3}|\sigma_\eps^2 u_\eps|^2 + \int_D |\sigma_\eps^2 u_\eps|^p  \bigr] \lesssim_p 1.
\end{align}
 Furthermore, for $\P$-almost every realization, the sequences $\{ \sigma_\eps^{2} u_\eps \}_{\eps >0}$ and $\{\sigma_\eps \nabla u_\eps \}_{\eps>0}$ are bounded in $L^p(D)$, $p \in (1; 2)$, and in $L^2(D)$, respectively.
\end{lem}

This, in turn, is a consequence of 
\begin{lem}\label{l.poincare}
For every  $p \in [1; 2]$ and for every $v \in H^1_0(D^\eps)$ we have 
\begin{align}\label{poincare.Lp}
\bigl( \int_D |\sigma_\eps v|^p \bigr)^{\frac 1 p} &\lesssim C_\eps(p) \bigl(\int_D |\nabla v|^2\bigr)^{\frac 1 2} \times \begin{cases}
1 \ \ &\text{for $p \in [1; 2)$}\\
|\log\eps|^3 \ \ \ &\text{if $p=2$,}
\end{cases}
\end{align}
where the random variables $\{ C_\eps(p) \}_{\eps> 0}$ satisfy
\begin{equation}
\begin{aligned}\label{rv.bounds}
&\limsup_{\eps \downarrow 0} C_\eps(p) \lesssim_p 1 \ \ \ &\text{$\P$-almost surely,}\\
&\limsup_{\eps \downarrow 0} \E \bigl[C_\eps^q(p) \bigr] \lesssim_p 1 \ \ &\text{for every  $q \in [1; +\infty)$.}
\end{aligned}
\end{equation}
\end{lem}

\smallskip

\begin{proof}[Proof of Lemma \ref{l.poincare}] 
As first step, we argue that the following Poincar\'e's inequality holds: Let $V$ be a convex domain. Assume that $V \subset B_r$ for some $r> 0$. Let $s < r$.  Then, for every $q \in [1; 2]$ and $u \in H^1(V \backslash B_s)$ such that $u=0$ on $\partial B_s$ it holds
\begin{align}\label{Poincare.easy}
\bigl(\int_{V \backslash B_s} |u|^q \bigr)^{\frac 1 q}\lesssim \frac{r^{\frac 3 q}}{s^{\frac 1 2}} \bigl(\int_{V \backslash B_s} |\nabla u|^2\bigr)^{\frac 1 2}.
\end{align}

The proof of this result is standard and may be easily proven by writing the integrals in spherical coordinates. We stress that the assumptions on $V$ allows to write the domain $V \backslash B_s$ as $\{ (\omega, r) \in \mathbb{S}^{n-1} \times \R_+, \ \  s \wedge R(\omega) \leq r < R(\omega) \}$ for some function $R: \mathbb{S}^{2} \to \R$ satisfying $\| R \|_{L^\infty(S^{2})} \leq r$.

\medskip

As second step, we construct an appropriate random tesselation for $D$: We consider the Voronoi tesselation $\{ V_z\}_{z\in \Phi}$ associated to the point process $\Phi$, namely the sets
$$
V_z :=\bigl\{ y \in R^3 \, \colon \, |y - z | = \min_{z \in \Phi} {|z-y|} \bigr\}, \ \ \ \ \text{for every $z\in \Phi$.}
$$
We define
$$
V_{\eps,z}:= \bigl\{ y 	\in \R^3 \, \colon \, \frac{1}{\eps} y \in V_z \bigr\}, \ \ \ \ A_\eps:= \bigl\{  z \in \Phi_\alpha \, \colon \, V_{z,\eps} \cap D \neq \emptyset \bigr\}.
$$
Note that, by the previous rescaling, we have that, if $\mathop{diam}(V_z):= r_z$, then $\mathop{diam}(V_{\eps,z})=\eps r_z$. 

\smallskip

It is immediate to see that, for every realization $\omega \in \Omega$, the sets $\{ V_{\eps,z} \}_{z\in A^\eps}$ are essentially disjoint, convex and cover the set $D$. Since $\Phi$ is stationary, the random variables $\{r_z\}_{z\in \Phi}$ are identically distributed. Furthermore, they are distributed as a generalized Gamma distribution having intensity $g(r)= C(\lambda) r^{8} \exp^{- c(d,\lambda) r^3}$ \cite{Moller.PVT}[Proposition 4.3.1.]. From this, it is a standard computation to show that 
\begin{align}\label{card.A.eps}
\lim_{\eps \downarrow 0}\eps^3 \E \bigl[ |\# A_\eps|^q\bigr]^{\frac 1 q} = |D| \ \ \ \ \text{for every $q \in [1, +\infty)$ }
\end{align}
and that there exists a constant $c=c(\lambda)>0$ such that for every function $F: \R_+ \to \R$ (that is integrable with respect to the measure $g(r) d r$) 
\begin{align}\label{size.voronoi}
\E \bigl[ \exp{(c r^3)}\bigr] \lesssim 1, \ \ \ \ \  |\E \bigl[ F(r_z) F(r_y) \bigr] - \E \bigl[ F(r) \bigr]^2 | \lesssim \E \bigl[ F(r)^4 \bigr]^{\frac 1 2} \eps^{-c |x- y|^3}.
\end{align}

\medskip

Equipped with $\{ V_{\eps,z} \}_{z \in A^\eps}$, we argue that for every realization of $H^\eps$ and all $p \in [1 ; 2)$ it holds
\begin{align}\label{poi.1}
\int_D |v|^p \leq \sigma_\eps^{-p}C_\eps(p) \bigl( \int_D |\nabla v|^2\bigr)^{\frac p 2}
\end{align}
with $C^\eps(p)^p:= \bigl(\eps^3 \sum_{z\in A^\eps} r_z^{\frac{6}{2-p}} \bigr)^{\frac{2-p}{2}}$. Note that by \eqref{card.A.eps}, \eqref{size.voronoi} and the Law of Large Numbers the family $\{ C^\eps(p)\}_{\eps >0}$ satisfies \eqref{rv.bounds}. We show \eqref{poi.1} as follows: For every $v\in H^1_0(D^\eps)$, we rewrite
\begin{align}
\int_D |v|^p = \sum_{z\in A^\eps} \int_{V_z^\eps} |v|^p.
\end{align}
Since $\rho_z \geq 1$, we have that $B_{\aeps}(\eps z) \subset B_{\aeps \rho_z}(\eps z)$ so that the function $v \in H^1_0(D^\eps)$ vanishes on $B_{\aeps}(\eps z)$. Hence, thanks to the choice of $\{ V_{\eps,z}\}_{z\in A^\eps}$, we apply Lemma \ref{Poincare.easy} in each set $V_z^\eps$ with $B_s= B_{\eps^\alpha}(\eps z)$ and $B_r= B_{\eps r_z}(\eps z)$ and infer that
\begin{align}\label{poi.1.a}
\int_D |v|^p \lesssim \eps^3 \eps^{-\frac{p}{2}\alpha} \sum_{z\in A^\eps} r_z^3 \bigl(\int_{V_z^\eps} |\nabla v|^2\bigr)^{\frac p 2}.
\end{align}
Since $p\in [1, 2)$, we may appeal to H\"older's inequality and conclude that
 \begin{align}
\int_D |v|^p \lesssim \sigma_\eps^{-p} \bigl(\eps^3 \sum_{z\in A^\eps} r_z^{\frac{6}{2-p}} \bigr)^{\frac{2-p}{2}} \bigl(\sum_{z \in A^\eps} \int_{V_z^\eps \cap D} |\nabla v|^2\bigr)^{\frac p 2},
\end{align}
i.e. inequality \eqref{poi.1}. This concludes the proof of \eqref{poincare.Lp} in the case $ p \in [1; 2)$. 

\smallskip

To tackle the case $p=2$ we need a further manipulation: we distinguish between points $z \in A^\eps$ having $r_z > - \log \eps$ or $r_z \leq - \log \eps$:
\begin{align}\label{poi.2}
\int_D |v|^2 = \sum_{z\in A^\eps \atop r_z \leq - \log \eps } \int_{V_z^\eps} |v|^2 +  \sum_{z\in A^\eps \atop r_z > - \log \eps } \int_{V_z^\eps} |v|^2.
\end{align}
We apply Poincar\'e's inequality in $H^1_0(D)$ on every integral of the second sum above. This implies that
\begin{align}
 \sum_{z\in A^\eps \atop r_z > -\log \eps } \int_{V_z^\eps \cap D} |v|^2 \lesssim \sigma_\eps^{-2} \int_D |\nabla v|^2 \bigl(\eps^3 \sum_{z \in A^\eps} \eps^{-3} \sigma_\eps^2 \1_{r_z > -\log \eps} \bigr),
\end{align}
so that Chebyschev's inequality and \eqref{size.voronoi} yield 
\begin{align}
 \sum_{z\in \Phi^\eps(D) \atop d_z > -\log \eps } \int_{V_z^\eps \cap D} |v|^2 \lesssim \sigma_\eps^{-2} C_\eps(2) \int_D |\nabla v|^2,
\end{align}
where we set $C_\eps(2) := \bigl(\eps^3 \sum_{z \in A^\eps} \exp\bigl( r_z^2 \bigr) \bigr)$. Note that, again by \eqref{card.A.eps}-\eqref{size.voronoi} and the Law of Large Numbers, this definition of $C_\eps(2)$ satisfies \eqref{rv.bounds}. Inserting the previous display into \eqref{poi.2} implies that
\begin{align}\label{poi.2}
\int_D |v|^2 \lesssim \sum_{z\in A^\eps \atop r_z \leq - \log \eps } \int_{V_z^\eps \cap D} |v|^2 + \sigma_\eps^2 C_\eps(2) \int_{D} |\nabla v|^2.
\end{align}
We now apply Lemma \ref{Poincare.easy} in the remaining sum and obtain \eqref{poi.1.a} with $p=2$, where the sum is restricted to the points $z \in A^\eps$ such that $r_z \leq -\log \eps$. From this, we infer that
\begin{align}\label{poi.2}
\int_D |v|^2 \lesssim \sigma_\eps^2 ( |\log\eps|^3  + C_\eps(2)^2) \int_{D} |\nabla v|^2.
\end{align}
By redefining $C_\eps(2)^2 = \min\bigl(\eps^3 \sum_{z \in A^\eps} \exp\bigl( r_z^2 \bigr) ; 1 \bigr)$, the above inequality immediately implies \eqref{poincare.Lp} for $p=2$. The proof of Lemma \ref{l.poincare} is complete.
\end{proof}

\begin{proof}[Proof of Lemma \ref{l.unif.bounds}]
We prove Lemma \ref{l.unif.bounds} for $u_\eps $ solving \eqref{P.eps.p}. The case \eqref{P.eps.s} is analogous. Since $f \in L^q(D)$ with $q \in (2 ; +\infty]$, we may test \eqref{P.eps.p} with $u_\eps$ and use H\"older's inequality to control
\begin{align}
\int_D |\nabla u_\eps |^2 \leq \bigl(\int_D |f|^q\bigr)^{\frac 1 q} \bigl(\int_D |u_\eps|^{\frac{q}{q-1}} \bigr)^{\frac{q-1}{q}}.
\end{align}
We thus appeal to \eqref{poincare.Lp} with $p= \frac{q}{q-1}$ and obtain that
\begin{align}\label{gradient.pw}
\bigl( \int_D |\sigma_\eps\nabla u_\eps |^2 \bigr)^{\frac 1 2} \lesssim C_\eps( \frac{q}{q-1})^{1- \frac{1}{q}} \bigl(\int_D |f|^q\bigr)^{\frac 1 q}.
\end{align}
Thanks to \eqref{rv.bounds} of Lemma \ref{l.poincare}, this yields that the sequence $\{ \sigma_\eps \nabla u_\eps \}_{\eps>0}$ is bounded in $L^2(D)$ for $\P$-almost every realization. Similarly, we infer \eqref{energy} by taking the expectation and applying H\"older's inequality.

\smallskip

We argue the remaining bounds for the terms of $u_\eps$ in a similar way: We combine Lemma \ref{l.poincare} with the same calculation above for \eqref{gradient.pw} and apply H\"older's inequality. This establishes Lemma \ref{l.unif.bounds}.
\end{proof}

\section{ Proof of Theorem \ref{t.main}, $(a)$}\label{s.thm.a}

\begin{lem}\label{l.oscillating.p}
There exists an $\eps_0=\eps_0(d)$ such that  for every $\eps < \eps_0$ and $\P$-almost every realization there exists a family $\{ w_\eps \}_{\eps > \eps_0} \subset W^{1,+\infty}(\R^3)$  such that $\|w_\eps  \|_{L^\infty(\R^3)} = 1$, $w_\eps = 0$ in $H^\eps$ and
\begin{align}\label{strong.conv.pointwise}
\limsup_{\eps \downarrow 0} \int_D | \sigma_\eps^{-1} \nabla w_\eps|^2 \lesssim 1, \ \ \ \lim_{\eps \downarrow 0}\int_D | w_\eps -1|^2 = 0.
\end{align}
In addition,
\begin{align}\label{strong.conv.expectations}
\limsup_{\eps \downarrow 0} \E\bigl[ \int_D |\sigma_\eps^{-2}\nabla w_\eps|^2 \bigr] \lesssim 1,  \ \ \ \ \ \lim_{\eps \downarrow 0} \E\bigl[ \int_D | w_\eps -1|^2 \bigr] =0,
\end{align}
and for every $\phi \in C^\infty_0(D)$ and $v_\eps \in H^1_0(D^\eps)$ satisfying the bounds of Lemma \ref{l.unif.bounds} and such that $\sigma_\eps^2 v_\eps \rightharpoonup v$ in $L^1(\Omega \times D)$, it holds
\begin{align}\label{conv.Delta}
\E\bigl[ |\langle -\Delta w_\eps ; v_\eps \phi \rangle -  k^{-1} \int_D v \phi| \bigr] \to 0.
\end{align}
Here, the constant $k$ is as in Theorem \ref{t.main}, $(a)$.
\end{lem}

\begin{proof}[Proof of Theorem \ref{t.main}, $(a)$]
The proof is similar to the one in \cite{Allaire_arma2}. We first show that $\sigma_\eps^2 u_\eps \rightharpoonup u$ in $L^p(D \times \Omega)$, $p\in [1, 2)$. By the uniform bounds of Lemma \ref{l.unif.bounds}, we have that, up to a subsequence, there exists a weak limit $u^* \in L^p(\Omega \times \R^d)$, $p \in [1,2)$. We prove that, $\P$-almost surely, the function $u^*= k f$ in $D$. This, in particular, also implies that the full family $\{\sigma_\eps^2u_\eps \}_{\eps>0}$ weakly converges to $u^*$.

\smallskip

We restrict to the converging subsequence $\{ \sigma_{\eps_j}^2 u_{\eps_j}\}_{j\in \N}$. However, for the sake of a lean notation, we forget about the subsequence $\{\eps_j \}_{j\in\N}$ and continue using the notation $u_\eps$ and $\eps \downarrow 0$.  Let $\eps_0$ and $\{w_\eps\}_{\eps >0}$ be as in Lemma \ref{l.oscillating.p}. For every $\eps < \eps_0$, $\chi \in L^\infty(\Omega)$ and $\phi \in C^\infty_0(D)$ we test equation \eqref{P.eps.p} with $\chi w_\eps \phi$ and take the expectation:
\begin{align}
\E \bigl[ \chi \int_D \nabla ( w_\eps \phi) \cdot \nabla u_\eps \bigr] = \E \bigl[\chi \int_D f w_\eps \phi \bigr].
\end{align}
Using Leibniz's rule, integration by parts and the bounds for $u_\eps$ and $w_\eps$ in Lemma \ref{l.unif.bounds} and \ref{l.oscillating.p} we reduce to
 \begin{align}
\lim_{\eps \downarrow 0}\E \bigl[ \chi  \langle -\Delta w_\eps ; u_\eps \phi \rangle \bigr] = \E \bigl[\chi \int_D f \phi \bigr].
\end{align}
We now appeal to \eqref{conv.Delta} in Lemma \ref{l.oscillating.p} applied to the converging subsequence $\{u_\eps \}_{\eps >0}$ and conclude that
\begin{align}
\E\bigl [\chi \int_D \phi (k^{-1}u^*- f) \bigr] = 0.
\end{align}
Since $\chi \in L^\infty(\Omega)$ and $\phi \in C^\infty_0(D)$ are arbitrary, we infer that for $\P$-almost every realization $u^*= k  f$ for (Lebesgue-)almost every $x \in D$. We stress that in this last statement we used the separability of $L^p(D)$, $p \in [1, \infty)$. This establishes that the full family $\sigma_\eps^2 u_\eps \rightharpoonup  k f$ in $L^p(\Omega \times D)$, $p \in [1, 2)$.

\bigskip

To conclude Theorem \ref{t.main}, $(a)$ it remains to upgrade the previous convergence from weak to strong. We fix $p \in [1, 2)$. By the assumption on $f$, the function $u^* \in L^q(D)$, for some $q \in (2; +\infty]$. Let $\{ u_n \}_{n\in \N} \subset C^\infty_0(D)$ be an approximating sequence for $u^*$ in $L^q(D)$.

\smallskip

Since $w_\eps \in W^{1,\infty}(D)$, the function $w_\eps u_n \in H^1_0(D)$. Hence, by Lemma \ref{l.poincare} applied to $u_\eps - w_\eps u_n$ we obtain
\begin{align}
\E\bigl[ \int_D |\sigma_\eps^2 u_\eps - w_\eps u_n |^p \bigr] \leq \sigma_\eps^{-p} \E \bigl[ C(p)^p \bigl( \int_D |\nabla( \sigma_\eps^2 u_\eps - w_\eps u_n)|^2 \bigr)^{\frac p 2}\bigr]
\end{align}
and, since $p< 2$ and $C(p)$ satisfies \eqref{rv.bounds} of Lemma \ref{l.poincare}, also
\begin{align}\label{strong.conv.b}
\E\bigl[ \int_D |\sigma_\eps^2 u_\eps - w_\eps u_n |^p \bigr] \leq \bigl( \sigma_\eps^{-2} \E \bigl[ \int_D |\nabla( \sigma_\eps^2 u_\eps - w_\eps u_n)|^2 \bigr]\bigr)^{\frac p 2}.
\end{align}
We claim that
\begin{align}\label{strong.conv.a}
\lim_{\eps \downarrow 0}  \sigma_\eps^{-2} \E \bigl[ \int_D |\nabla( \sigma_\eps^2 u_\eps - w_\eps u_n)|^2 \bigr] = k^{-1}\int_D |u_n - u^*|^2,
\end{align}
so that
\begin{align}\label{strong.conv.b}
\limsup_{\eps \downarrow 0}\E\bigl[ \int_D |\sigma_\eps^2 u_\eps - w_\eps u_n |^p \bigr] \lesssim \int_D |u_n - u^*|^2.
\end{align}
Provided this holds, we establish Theorem \ref{t.main}, $(a)$, as follows: By the triangle inequality we have that
\begin{align}
\int_D |\sigma_\eps^{-2}u_\eps - u|^p \leq \int_D |u_n - u^*|^p  + \int_D |\sigma_\eps^{-2}u_\eps-  w_\eps u_n |^p + \int_D |w_\eps - 1|^p |u_n|.
\end{align}
Since $u^*$ and $u_n \in C^\infty_0(D)$ are deterministic, we take the expectation and use Lemma \ref{l.oscillating.p} with \eqref{strong.conv.b} to get
\begin{align}
\limsup_{\eps \downarrow 0} \E\bigl[ \int_D |\sigma_\eps^2 u_\eps - u|^p \bigr] \lesssim \int_D |u_n - u^*|^p + ( \int_D |u_n - u^*|^2)^{\frac p 2}.
\end{align}
This implies the statement of Theorem \ref{t.main}, $(a)$, since $p< 2$ and $\{ u_n \}_{n\in \N}$ converges to $u$ in $L^2(D)$.

\bigskip

We thus turn to \eqref{strong.conv.a}: We skip the lower index $n \in \N$ and write $u$ instead of $u_n$. If we expand the inner square, we write
\begin{align}\label{strong.1}
\sigma_\eps^{-2} \E \bigl[ \int_D |\nabla( \sigma_\eps^2 u_\eps - w_\eps u)|^2 \bigr]= \sigma_\eps^2 \E\bigl[ \int_D |\nabla u_\eps|^2 \bigr]- 2 \E \bigl[\int_D \nabla u_\eps \cdot \nabla (w_\eps u) \bigr]+ \sigma_\eps^{-2}\E \bigl[\int_D|\nabla(w_\eps u)|^2\bigr]  .
\end{align}
For first term in the right-hand we use \eqref{P.eps.p} and the fact that $\sigma_\eps^2 u_\eps \rightharpoonup u^*$ in $L^p(\Omega \times D)$ with $p \in [1, 2)$. Hence,
\begin{align}\label{strong.0}
\lim_{\eps \downarrow 0} \sigma_\eps^2 \E \bigl[ \int_D |\nabla u_\eps|^2 \bigr]=  \int_D f u^* .
\end{align}
We focus on the remaining two terms in \eqref{strong.1}: Using Leibniz's rule and an integration by parts we have that
\begin{align}
\E \bigl[\int_D \nabla u_\eps \cdot \nabla (w_\eps u)\bigr] = \E\bigl[ \int_D w_\eps \nabla u_\eps \cdot \nabla u \bigr] + \E \bigl[\langle -\Delta w_\eps : u_\eps u \rangle \bigr] - \E \bigl[\int_D u_\eps \nabla w_\eps \cdot \nabla u \bigr] .
\end{align}
Thanks to Lemma \ref{l.unif.bounds}, Lemma \ref{l.oscillating.p} and since $u \in C^{\infty}_0(D)$, the first and second term vanish in the limit $\eps \downarrow 0$. Hence,
\begin{align}\label{strong.2}
\lim_{\eps \downarrow 0}  \E \bigl[ \int_D \nabla u_\eps \cdot \nabla (w_\eps u) \bigr]=  \lim_{\eps \downarrow 0} \E \bigl[ \langle -\Delta w_\eps ; u_\eps u \rangle \bigr].
\end{align}
By Lemma \ref{l.unif.bounds} and since $u_\eps \rightharpoonup u^*$, we may apply  \eqref{conv.Delta} of Lemma \ref{l.oscillating.p} with $\phi= u$ and $v_\eps = u_\eps$ to the limit on the right-hand side above. This yields
\begin{align}\label{strong.2}
\lim_{\eps \downarrow 0}  \E \bigl[ \int_D \nabla u_\eps \cdot \nabla (w_\eps u) \bigr]=  \int k^{-1} u^* u.
\end{align}

\smallskip

We now turn to the last term in \eqref{strong.1}. Also here, we use Leibniz rule to compute
\begin{align}
 \sigma_\eps^{-2}\E \bigl[\int_D |\nabla(w_\eps u)|^2\bigr]=  \sigma_\eps^{-2}\biggl( \E \bigl[ \int_D |\nabla w_\eps|^2 u^2 \bigr] + \E \bigl[\int_D |\nabla u|^2 w_\eps^2\bigr] + 2\E \bigl[ \int_D u \, w_\eps \, \nabla w_\eps \cdot \nabla u \bigr]\biggr).
\end{align}
By an argument similar to the one for \eqref{strong.2}, we reduce to
\begin{align}
\lim_{\eps \downarrow 0} \sigma_\eps^{-2}\E\bigl[ \int_D |\nabla(w_\eps u)|^2\bigr] = \lim_{\eps \downarrow 0} \sigma_\eps^{-2}\E\bigl[ \langle -\Delta w_\eps ; w_\eps u^2 \rangle.
\end{align}
We now apply \eqref{conv.Delta} of Lemma \ref{l.oscillating.p} to $v_\eps = w_\eps u$ and $\phi=u$. This implies that
\begin{align}\label{strong.3}
\lim_{\eps \downarrow 0} \sigma_\eps^{-2}\E\bigl[ \int_D |\nabla(w_\eps u)|^2\bigr] = \int k^{-1} u^2 .
\end{align}

\smallskip

Inserting \eqref{strong.0}, \eqref{strong.2} and \eqref{strong.3} into \eqref{strong.1} we have that
\begin{align}\label{strong.a}
\lim_{\eps \downarrow 0} \E \bigl( \int_D |\sigma_\eps u_\eps - w_\eps u|^q\bigr)^{\frac 2 q} = \int_D f u^*  + \int_D k^{-1} u^2 - 2 k^{-1} \int_D u^* u.
\end{align}
Since $u^* = k f$, it is easy to see the the right-hand side above equals the right-hand side of \eqref{strong.conv.a}. This establishes \eqref{strong.conv.a} and concludes the proof of Theorem \ref{t.main}, case $(a)$.
\end{proof}

\subsection{Proof of Lemma \ref{l.oscillating.p}}
 Lemma \ref{l.oscillating.p} may be proven in a way that is similar to \cite{GHV}[Lemma 3.1]. The first crucial ingredient is the following lemma, that allows to find a suitable partition of the holes $H^\eps$ by dividing this set into a part containing well separated holes and another one containing the clusters. The next result is the analogue of \cite{GHV}[Lemma 4.2] with the different rescaling of the radii of the balls generating the set $H^\eps$.

\smallskip

For every $x\in \R^3$, we recall the definition of $R_{\eps,x}$ in \eqref{distance}. We have:

\begin{lem}\label{l.geometry.p} Let $\gamma \in (0, \alpha -1)$. Then there exists a partition $H^\eps:= H^\varepsilon_{g} \cup H^\varepsilon_{b}$, with the following properties:
\begin{itemize}
\item There exists a subset of centres $n^\eps(D) \subset \Phi^\eps(D)$ such that
\begin{align}\label{good.set.ppp}
H^\varepsilon_g : = \bigcup_{z \in n^\varepsilon(D)} B_{\aeps \rho_z}( \varepsilon z ), \ \ \ \ \min_{z \in n^\eps(D)}R_{\eps, z}\geq \eps^{1+\frac{\gamma}{2}}, \ \ \ \ \ \max_{z \in n^\eps(D)}\aeps \rho_z \leq \eps^{1+\gamma}.
\end{align}

\smallskip

\item There exists a set $D^\eps_b(\omega) \subset  \R^3$ satisfying
\begin{align}
H^\varepsilon_{b} \subset D^\eps_b, \ \ \ \capacity ( H^\varepsilon_b, D_b^\eps) \lesssim C(\gamma) \aeps \sum_{z \in \Phi^\eps(D) \backslash n^\eps(D)} \rho_z \label{capacity.sum}
\end{align}
and for which
\begin{align}\label{ppp.distance.good.bad}
B_{\frac{R_{\eps,z}}{2}}(\eps z) \cap D^\eps_b = \emptyset, \ \ \  \ \ \ \ \text{for every $z \in n^\eps(D)$.} 
\end{align}
\end{itemize}
Finally, we have that
\begin{align}\label{bad.cap.vanishes}
\lim_{\eps \downarrow 0}\eps^{3}\sum_{z \in \Phi^\eps(D) \backslash n^\eps(D)} \rho_z^{\frac 3 \alpha} = 0, \ \ \text{$\P$-almost surely},  \ \ \ \ \ \lim_{\eps \downarrow 0}\E \bigl[\eps^{3}\sum_{z \in \Phi^\eps(D) \backslash n^\eps(D)} \rho_z^{\frac 3 \alpha}\bigr] = 0.
\end{align}
\end{lem}

Let $\gamma$ in Lemma \ref{l.geometry.p} be fixed. We construct $w_\eps$ as done in \cite[]{GHV}: we set $w_\eps = w_\eps^g \wedge w_\eps^b$ with 
\begin{align}\label{def.w}
w^\eps_b:= \begin{cases} 1- \mathop{argmin}\capacity(H^\eps_b ; D^\eps_b) \ \ &\text{in $D^\eps_b$}\\
1 \ \ &\text{in $\R^3\backslash D^\eps_b$} 
\end{cases} \ \ \ 
w^{\eps}_g = \begin{cases}
w_{\eps,z} \ \ &\text{in $B_{R_{\eps,z}}(\eps z),  z\in n^\eps(D)$}\\
1 \ \ &\text{in $\R^3 \backslash \bigcup_{z\in n^\eps(D)} B_{R_{\eps,z}}(\eps z)$}
\end{cases}
\end{align}
where for each $z\in n^\eps(D)$, the function $w_{\eps,z}$ vanishes in the hole $B_{\aeps \rho_z}(\eps z)$ and solves
\begin{align}\label{harmonic.cell}
w^{\eps}_g = \begin{cases}
-\Delta w_{\eps,z} = 0 \ \ \ &\text{in $B_{R_{\eps,z}}(\eps z)\backslash B_{\aeps\rho_z}(\eps z)$}\\
0 \ \ \ &\text{on $\partial B_{\aeps\rho_z}(\eps z)$}\\
1 \ \ \ &\text{on $\partial B_{R_{\eps,z}}(\eps z)$}
\end{cases}
\end{align}
We also define of the measure 
\begin{align}\label{def.mu.eps}
\mu_\eps= \sum_{z\in n^\eps(D)} \partial_n w_{\eps,z} \delta_{\partial B_{R_{\eps,z}}(\eps z)} \in H^{-1}(D).
\end{align}
We stress that all the previous objects depend on the choice of the parameter $\gamma$ in Lemma \ref{l.geometry.p}. The next result states that this parameter may be chosen in so that the norm $\| \mu_\eps - 4\pi \lambda \E \bigl[ \rho \bigr] \|_{H^{-1}(D)}$ is suitably small. This, together with Lemma \ref{l.geometry.p}, provides the crucial tool to show Lemma \ref{l.oscillating.p}:

\smallskip

\begin{lem}\label{conv.measure}
There exists $\gamma \in (0, \alpha-1)$ such that if $\mu_\eps$ is as in \eqref{def.mu.eps} there exists $\kappa >0$ such that for every random field $v \in  H^1_0(D)$ 
\begin{align}
\E \bigl[ \langle (\sigma_\eps^{-2}\mu_\eps - 4\pi \E\bigl[\rho \bigr]); v \rangle \bigr] \lesssim \eps^{\kappa}\bigl(\sigma_\eps^{-1}\E\bigl[ \int_D |\nabla v|^2 \bigr]^{\frac 1 2} +  \E\bigl[ \int_D |v|^2 \bigr]^{\frac 1 2}\bigr).
\end{align}
\end{lem}

\bigskip

\begin{proof}[Proof of Lemma \ref{l.oscillating.p}]
By construction, it is clear that, for $\P$-almost every realization, the functions $w_\eps \in W^{1,\infty}(\R^3) \cap H^1(\R^3)$, vanish in $H^\eps$ and are such that $\| w_\eps \|_{L^\infty(\R^3)}=1$.

\smallskip

We now turn to \eqref{strong.conv.pointwise}. Using the definitions of $w_g^\eps$ and $w^\eps_b$ and Lemma \ref{l.geometry.p} we have that
\begin{align}\label{L2.dec}
\| w_\eps - 1 \|_{L^2(D)} = \| w_\eps^g - 1 \|_{L^2(D)} + \| w_\eps^b - 1\|_{L^2(D)}.
\end{align}
By Poincar\'e's inequality in each ball $\{ B_{R_{\eps,z}}(\eps z) \}_{z\in n^\eps(D)}$ we bound
\begin{align}\label{w.g.L2}
\| w_\eps^g - 1 \|_{L^2(D)}^2 \leq \sum_{z\in n^\eps(D)} \eps^2 \| \nabla w^\eps_g \|_{L^2(D)}^2 \lesssim \eps^{\alpha-1} \eps^3 \sum_{z \in n^\eps(D)} \rho_z.
\end{align}
Thanks to \eqref{integrability.p} and the Strong law of Large numbers, for $\P$-a.e. realization the right-hand side vanishes in the limit $\eps \downarrow 0$.

\smallskip

We now turn to the second term: Since by the maximum principle $|w_\eps^b -1| \leq 1$, we may use the definition of $D^\eps_b$ to bound
\begin{align*}
 \| w_\eps^b - 1\|_{L^2(D)}^2 &\leq |D_\eps^b \cap D| \leq \sum_{z\in \Phi^\eps(D)\backslash n^\eps} \eps^{3\alpha} (\rho_z \wedge \eps^{-\alpha})^3 \lesssim \eps^3\sum_{z\in \Phi^\eps(D)\backslash n^\eps} \rho_z^{\frac{3}{\alpha}} .
 \end{align*}
Thanks to \eqref{bad.cap.vanishes} in Lemma \ref{l.geometry.p}, the right-hand side vanishes in the limit $\eps \downarrow 0$ for $\P$-almost every realization. Combining this with \eqref{w.g.L2} and \eqref{L2.dec} yields \eqref{strong.conv.pointwise} for  $w_\eps - 1$. Inequality \eqref{strong.conv.pointwise} for  $\sigma_\eps^{-1}\nabla w_\eps$ follows by Lemma \ref{l.geometry.p} and the definition \eqref{def.w} of $w_\eps$ as done in \cite{GHV}[Lemma 3.1].  Limit \eqref{strong.conv.expectations} may be argued as done above for \eqref{strong.conv.pointwise}, this time appealing to the bound \eqref{integrability.p} and the stationarity of $(\Phi, \mathcal{R})$.

\bigskip

It thus remains to show \eqref{conv.Delta}.  Using \eqref{def.w}, \eqref{def.mu.eps} and the fact that $\phi u_\eps \in H^1_0(D^\eps)$,  we may decompose
\begin{align}\label{decomposition.Delta}
\langle -\Delta w_\eps; \phi v_\eps \rangle = \langle \mu_\eps ; \phi v_\eps \rangle +  \int_D \nabla w_\eps^b \cdot \nabla(\phi v_\eps).
\end{align} 
Since $v_\eps$ is assumed to satisfy the bounds in Lemma \ref{l.unif.bounds}, H\"older's inequality, Lemma \ref{l.unif.bounds} , definition \eqref{def.w} and \eqref{bad.cap.vanishes} of Lemma \ref{l.geometry.p} imply that
\begin{align}
\lim_{\eps \downarrow 0}\E\bigl[  |\int_D \nabla w_\eps^b \cdot \nabla(\phi v_\eps)|\bigr]\leq \lim_{\eps \downarrow 0}\E \bigl[ \capacity( H^\eps_b ; D^\eps_b) \bigr] = 0.
\end{align}
This and \eqref{decomposition.Delta} thus yield that
\begin{align}
 \limsup_{\eps \downarrow 0}\E\bigl[  |\langle -\Delta w_\eps; \phi v_\eps \rangle- k^{-1} \int_D v \phi |\bigr] =  \limsup_{\eps \downarrow 0}\E\bigl[  |\langle \mu_\eps ; \phi v_\eps \rangle- k^{-1} \int_D v \phi |\bigr].
\end{align}
Using the triangle inequality and the assumption $v_\eps \rightharpoonup v$ in $L^1(\Omega \times D)$, we further reduce to
\begin{align}\label{error.term}
\limsup_{\eps \downarrow 0}\E\bigl[  |\langle -\Delta w_\eps; \phi v_\eps \rangle- k^{-1} \int_D v \phi |\bigr]  = \limsup_{\eps \downarrow 0}\E\bigl[  |\langle (-\sigma_\eps^2 \Delta w_\eps -  k^{-1} ; \phi \sigma_\eps^{-2}v_\eps \rangle |\bigr]
\end{align}
By Lemma \ref{conv.measure}, there exists $\kappa > 0$ such that
\begin{align}
\limsup_{\eps \downarrow 0}\E\bigl[  |\langle (-\sigma_\eps^2 \Delta w_\eps -  4\pi \lambda \E\bigl[ \rho \bigr]); \phi \sigma_\eps^{-2}v_\eps \rangle |\bigr]&\leq \limsup_{\eps \downarrow 0} \eps^{\kappa} \bigl( \sigma_\eps^{-1} \E\bigl[ \int_{D}|\nabla(\phi \sigma_\eps^2 v_\eps)|^2 \bigr]^{\frac 1 2} + \E\bigl[ \int_{D}(\phi \sigma_\eps^2 v_\eps)^2 \bigr]^{\frac 1 2} \bigr).
\end{align}
Thanks to the assumptions on $v_\eps$, we infer that the right-hand side is zero. This, together with \eqref{error.term}, yields \eqref{conv.Delta}. The proof of Lemma \ref{l.oscillating.p} is thus complete.
\end{proof}

\bigskip

\begin{proof}[Proof of Lemma \ref{conv.measure}]
We divide the proof into steps. The strategy of this proof is similar to the one for \cite{G}[Theorem 2.1, (b)]. 

\bigskip

\noindent \textit{ Step 1: (Construction of a partition for D) } Let $Q := [-\frac 1 2 ;\frac 1 2]^3$; for $k \in \N$ and $x\in \R^3$ we define 
$$
Q_{\eps, k,x}:= \eps z + k\eps Q, \ \ \ \ Q_{\eps,x} := Q_{\eps, 1,x}
$$
Let $N_{k,\eps} \subset \Z^3$ be a collection of points such that $|N_{k,\eps}| \lesssim \eps^{-3}$ and $D \subset \bigcup_{x \in N_{k,\eps}} Q_{\eps, k,x}$. For each $x \in N_{k,\eps}$ we consider the collection of points $N_{\eps, k, x}:= \{ z \in n^\eps(D) \, \colon \, \eps z \in Q_{\eps, k,x} \} \subset \Phi^\eps(D)$ and define the set
\begin{align}\label{covering.Poisson}
K_{\eps, k, x} := \bigl(Q_{\eps, k, x} \bigcup_{z \in N_{\eps, k,x}}  Q_{\eps,z} \bigr) \backslash \bigcup_{z \in \tilde\Phi^\eps(D) \backslash N_{\eps, k,x}} Q_{\eps,z}.
\end{align}
Since by definition of $n^\eps(D)$ in Lemma \ref{l.geometry.p} the cubes $\{ Q_{\eps,z} \}_{z \in \tilde\Phi^\eps(D)}$ are all disjoint, we have that
\begin{equation}
\begin{aligned}\label{properties.covering}
&D \subset \bigcup_{x \in N_{\eps, k}} K_{\eps, k,x}, \ \ \  \sup_{x\in N_{\eps,k}}|\mathop{diam}(K_{\eps, k,x})| \lesssim k \eps,\\
&( k  - 1)^3 \eps^3 \leq |K_{k,x}| \leq ( k + 1)^3 \eps^3 \ \ \ \text{for every $x \in N_{\eps,k}$.}
\end{aligned}
\end{equation}
Note that the previous properties hold for every realization $\omega \in \Omega$. 

\bigskip

\noindent \textit{ Step 2.}  For $k \in \N$ fixed, let $\{ K_{\eps, x,k} \}_{x \in N_{k,\eps}}$ be the covering of $D$ constructed in the previous step. We define the random variables  
\begin{align}\label{averaged.sum}
S_{\eps, k,x}:= \frac{4\pi}{|K_{\eps, x,k}|}\sum_{z \in N_{\eps, k,x}} Y_{\eps, z} \ \ \ \ \ Y_{\eps,z}:= \eps^3 \rho_z \frac{R_{\eps,z}}{R_{\eps,z} - \aeps \rho_z}.
\end{align}
and construct the random step function
\begin{align}\label{mu.k.eps}
m_\eps(k) = 4\pi \sum_{x \in N_{\eps, k}} S_{\eps, k,x}  \1_{K_{\eps, k, x}}.
\end{align}

\smallskip

Let $v$ be as in the statement of the lemma and $m_\eps(k)$ as above. The triangle and Cauchy-Schwarz inequalities imply that
\begin{align}\label{triangle.mu}
\E \bigl[ \langle& \sigma_\eps^{-2}\mu_\eps - 4\pi \lambda \E\bigl[ \rho \bigr] ; v \rangle \bigr]\\
& \leq \E \bigl[ \|\sigma_\eps^{-2} \mu_\eps - m_\eps(k) \|_{H^{-1}}^2 \bigr]^{\frac 1 2} \E\bigl[ \|\nabla v \|_{L^2(D)}^2 \bigr]^{\frac 1 2}  + \E \bigl[ \|m_\eps(k)  - 4\pi \lambda \E\bigl[ \rho \bigr] \|_{L^2}^2 \bigr]^{\frac 1 2} \E\bigl[ \| v \|_{L^2(D)}^2 \bigr]^{\frac 1 2},
\end{align}
so that the proof of the lemma reduces to estimating the norms 
$$
\E \bigl[ \|\sigma_\eps^{-2} \mu_\eps - m_\eps(k) \|_{H^{-1}}^2 \bigr]^{\frac 1 2}, \ \ \ \ \ \E \bigl[ \|m_\eps(k)  - 4\pi \lambda \E\bigl[ \rho \bigr] \|_{L^2}^2 \bigr]^{\frac 1 2}.
$$ 

\smallskip

We now claim that there exists a $\gamma >0$, $k \in \N$ 
\begin{align}\label{estimates.mu.muk}
\E\bigl[ \|\sigma_\eps^{-2}\mu_\eps - m_\eps(k) \|_{H^{-1}(D)}^2 \bigr] &\lesssim \eps^\kappa \sigma_\eps^{-2}, \ \ \ \ \E\bigl[ \| m_\eps(k) - 4\pi \lambda \E\bigl[ \rho \bigr]\|_{L^2(D)}^2 \bigr] \lesssim  \eps^\kappa
\end{align}
for a positive exponent $\kappa >0$. Combining these two inequalities with \eqref{triangle.mu} establishes Lemma \ref{conv.measure}.

\bigskip

In the remaining part of the proof we tackle inequalities \eqref{estimates.mu.muk}. We follow the same lines of \cite{G}[Theorem 1.1, (b)]. and thus only sketch the main steps for the argument. 

\bigskip

\noindent \textit{Step 3.} We claim that
\begin{align}\label{mu.muk.a}
 \E\bigl[ \|\sigma_\eps^{-2}\mu_\eps - m_\eps(k) \|_{H^{-1}(D)}^2 \bigr] \lesssim (k \eps)^2  |\log\eps| \eps^{-(\alpha -1-\gamma)(2- \frac 3 \alpha)_+}.
\end{align}

\smallskip

 We first argue that that
\begin{align}\label{mu.eps.a}
 \| \sigma_{\eps}^{-2}\mu_\eps -m_\eps(k) \|_{H^{-1}(D)}^2 {\lesssim} (\eps k)^2 \eps^{3} \sum_{z \in n^\eps(D)}\rho_z^2 (\eps d_z)^{-3},
\end{align}
This follows by  Lemma \ref{Kohn_Vogelius.general} applied to the measure $\sigma_\eps^{-2}\mu_\eps$: In this case, the random set of centres is $\mathcal{Z}=\tilde\Phi^\eps(D)$, the random radii $\mathcal{R}= \{ R_{\eps,z}\}_{z\in \tilde \Phi(D)}$, the functions $g_i = \sigma_{\eps}^{-2} \nabla_\nu w_{\eps,z}$, $z\in \tilde\Phi^\eps(D)$ and the partition $\{ K_{\eps,k,x}\}_{x \in N_{\eps,k}}$ of the previous step. Note that, by construction, this partition satisfies the assumptions of Lemma \ref{Kohn_Vogelius.general}.  The explicit formulation of the harmonic functions $\{ w_{\eps,z}\}_{z\in n^\eps(D)}$ defined in \eqref{harmonic.cell} (c.f. also \cite{G}[(2.24)]) implies that for every $z \in n^\eps(D)$
\begin{align}\label{bounds.w.eps.z}
\int_{\partial B_{R_{\eps,z}}(\eps z)} |\sigma_{\eps}^{-2}\partial_\nu w_{\eps,z}|^2 \lesssim \eps^3 \rho_z^2  d_z^{-3}, \ \ \ \int_{\partial B_{\eps,z}} \sigma_{\eps}^{-2} \partial_\nu w_{\eps,z} \stackrel{\eqref{averaged.sum}}{=} Y_{\eps,z}. 
\end{align}
Therefore, Lemma \ref{Kohn_Vogelius.general} and the bounds \eqref{bounds.w.eps.z} yield that 
\begin{align}
 \| \sigma_{\eps}^{-2}\mu_\eps -m_\eps(k) \|_{H^{-1}(D)}^2{\lesssim}\sup_{x \in N_{k,\eps}}\text{diam}(K_{\eps,k,x} )\sum_{z \in \tilde\Phi^\eps(D)}\rho_z^2 (\eps d_z)^{-3},
\end{align}
which implies \eqref{mu.eps.a} thanks to \eqref{properties.covering}.

\smallskip

It thus remains to pass from \eqref{mu.eps.a} to \eqref{mu.muk.a}: We do this by taking the expectation and arguing as for \cite{G}[Inequality (4.22)]. We rely on the stationarity of $(\phi, \mathcal{R})$, the properties of the Poisson point process and the fact that $z \in n_\eps$ implies that $\aeps \rho_z \leq \eps^{1 + \gamma}$ and $R_{\eps,z} \geq \eps^{1+\frac 1 2\gamma}$.

\bigskip

\textit{Step 4.} We now turn to the left-hand side in the second inequality of \eqref{estimates.mu.muk} and show that
\begin{align}\label{mu.muk.b}
\E\biggl[ \| m_\eps(k) - 4\pi \lambda \E \bigl[ \rho \bigr] \|_{L^2(D)}^2 \biggr] &\leq k^{-3}\eps^{-(\alpha-1 - \gamma)} + k^{-1} + \eps^{2\gamma}\\
& \quad \quad + \eps^{(\alpha-1-\gamma)(\frac3\alpha-1)} + \eps^{4(1+\gamma)-\alpha} + \eps k \eps^{-(\alpha-1-\gamma)(2-\frac 3 \alpha)_+}.
\end{align}

\smallskip

 The proof of this step is similar to \cite{G}[Theorem 2.1, (b)]: Using the explicit formulation of $m_\eps(k)$ we reduce to
\begin{align}\label{first.reduction.m}
\E\biggl[ \| m_\eps(k) - 4\pi \lambda \E \bigl[ \rho \bigr] \|_{L^2(D)}^2 \biggr] \lesssim \avsum_{x \in N_{k,\eps}} \E \bigl[ (S_{k,\eps,x} - \lambda \E\bigl[ \rho \bigr])^2 \bigr]
\end{align}
If $\mathring{N}_{\eps,k}:= \{ x \in N_{\eps,k} \, \colon \, \text{dist}(Q_{\eps,k, x} ; \partial D) > 2\eps \}$, we split
\begin{align}\label{control.reduction}
\avsum_{x \in N_{k,\eps}} \E \bigl[ (S_{k,\eps,x} - \lambda \E\bigl[ \rho \bigr])^2 \bigr]  \lesssim (\eps k)^{3}\sum_{x \in N_{k,\eps}\backslash \mathring{N}_{\eps,k} } \E \bigl[ (S_{k,\eps,x} - \lambda \E\bigl[ \rho \bigr])^2 \bigr] + \avsum_{x \in \mathring{N}_{k,\eps}} \E \bigl[ (S_{k,\eps,x} - \lambda \E\bigl[ \rho \bigr])^2 \bigr].
\end{align}
Since $\partial D$ is $C^1$ and compact, for $\eps$ small enough (depending on $D$) we have
\begin{align}\label{mu.k.b}
 (\eps k)^{3}\sum_{x \in N_{k,\eps}\backslash \mathring{N}_{\eps,k} } \E \bigl[ (S_{k,\eps,x} - \lambda \E\bigl[ \rho \bigr])^2 \bigr]\lesssim \eps k \E \bigl[ \rho^2 \1_{\aeps\rho < \eps^{1+\gamma}} \bigr] \stackrel{\eqref{integrability.p}}{\lesssim} \eps k \eps^{-(\alpha-1-\gamma)(2-\frac 3 \alpha)_+}.
\end{align}
By stationarity, the second term in \eqref{control.reduction} is controlled by 
\begin{align}\label{clt}
\E \bigl[ (S_{k,\eps,0} - \lambda \E\bigl[ \rho \bigr])^2 \bigr].
\end{align}
Hence,
\begin{align}
\avsum_{x \in N_{k,\eps}} \E \bigl[ (S_{k,\eps,x} - \lambda \E\bigl[ \rho \bigr])^2 \bigr]  \lesssim \E \bigl[ (S_{k,\eps,0} - \lambda \E\bigl[ \rho \bigr])^2 \bigr] +  \eps k \eps^{-(\alpha-1-\gamma)(2-\frac 3 \alpha)_+}.
\end{align}

The remaining term on the righ-hand side may be controlled by the right-hand side in \eqref{mu.muk.b} by means of standard CLT arguments as done in \cite{G}[Inequality (4.23)] for the analogous term. We stress that the crucial observation is that the random variables $S_{k,\eps,x} - \lambda \E\bigl[ \rho \bigr]$ are centred up to an error term. We mention that in this case the set $K_{\eps,k,x}$ has been defined in a different way from \cite{G} and we use properties \eqref{properties.covering} instead of \cite{G}[(4.13)]. This yields \eqref{mu.muk.b}

\bigskip

\textit{Step 5.} We show that, given \eqref{mu.muk.a} and \eqref{mu.muk.b} of the previous two steps, we may pick $\gamma$ and $k\in \N$ such that inequalities \eqref{estimates.mu.muk} hold: Thanks to the definition of $\sigma_\eps$ and since $\alpha \in (1; 3)$, we may find $\gamma$ close enough to $\alpha -1$, e.g. $\gamma= \frac{20}{21}(\alpha-1)$, and a $k \in \N$, e.g. $k=-\frac{9}{20}(\alpha-1)$, such that 
$$
(\eps k) |\log\eps| \eps^{-(\alpha -1-\gamma)(2 -\frac 3 \alpha)_+} \leq \sigma_\eps^{-2}k^2 \eps^{(\alpha-1-\gamma)}\leq \eps^{\frac 1 20 (\alpha -1)}.
$$
This, thanks to \eqref{mu.muk.a}, implies that the first inequality in \eqref{estimates.mu.muk} holds with the choice $\kappa=\frac 1 20 (\alpha-1) > 0$. The same values of $\gamma$ and $\kappa$ yield that also the right hand side of \eqref{mu.muk.b} is bounded by $\eps^{1-\frac 1 2 (\alpha-1)}$. This yields also the remaining inequality in \eqref{estimates.mu.muk} and thus concludes the proof of  Lemma \ref{conv.measure}.
\end{proof}

\begin{proof}[Proof of Lemma \ref{l.geometry.p}]
The proof of this lemma follows the same construction implemented in the proof of \cite{G}[Lemma 4.1] with $d=3$, $\delta = \gamma$ and with the radii $\{\rho_z \}_{z \in \Phi^\eps(D)}$ rescaled by $\aeps$ instead of $\eps^3$. Note that the constraint for $\gamma$ is due to this different rescaling. In the current setting, we replace $\eps^2$ by $\eps^{1+\frac 1 2 \gamma}$ in the definition of the set $K_b^\eps$ in \cite{G}[(4.7)]. Estimate  \eqref{bad.cap.vanishes} may be argued as \cite{G}[Lemma 4.4] by relying on \eqref{integrability.p}.
\end{proof}

\section{Proof of Theorem \ref{t.main}, $(b)$}\label{s.thm.b}
The next lemma is the analogue of Lemma \ref{l.oscillating.p}:

\begin{lem}\label{l.reduction.2}
For every $\delta > 0$, there exists an $\eps_0 >0$ and a set $A_\delta \in \mathcal{F}$, having  $\P(A_\delta) \geq 1-\delta$, such that for every $\omega \in A_\delta$ and $\eps \leq \eps_0$ there exists a linear map
$$
R_\eps : \{ \phi \in C^\infty_0(D, \Rd) \, \colon \, \nabla \cdot \phi = 0 \}  \to H^1_0(D, \Rd)
$$
satisfying $R_\eps \phi =0$ in $H^\eps$,  $\nabla \cdot R_\eps \phi =0$ in $D$ and such that
\begin{align}\label{aver.R}
\limsup_{\eps \downarrow 0}\E \bigl[ \1_{\mathcal{A}_{\delta} } \int_D |\sigma_\eps^{-1} \nabla R_\eps(\phi)|^2 \bigr] \lesssim \| v\|_{C^1(D)}^2, \ \ \ \E \bigl[ \1_{\mathcal{A}_{\delta} } \int_D |R_\eps(\phi) - \phi|^2 \bigr] \to 0.
\end{align}
Furthermore, if $v_\eps$ satisfies the bounds of Lemma \ref{l.unif.bounds} and $\sigma_\eps^2 v_\eps \rightharpoonup v$ in $L^1(\Omega \times D)$, then
\begin{align}\label{meas.R}
\E \bigl[ \1_{\mathcal{A}_{\delta} } | \int \nabla R_\eps(\phi) \cdot \nabla v_\eps - K^{-1}\int \rho v | \, \bigr] \to 0.
\end{align}
\end{lem}

\begin{proof}[Proof of Theorem \ref{t.main}, $(b)$]
The proof of this statement is very similar to the one for case $(a)$ and we only emphasize the few technical differences: Using the bounds of Lemma \ref{l.unif.bounds}, we have that, up to a subsequence, $u_{\eps_j} \rightharpoonup u^*$ in $L^p(\Omega \times D)$, $1 \leq p < 2$. We prove that 
\begin{align}\label{weak.convergence.stokes}
u^*=K(f - \nabla p), 
\end{align}
 where $p \in H^{1}(D)$ is the unique weak solution to 
\begin{align}\label{eq.P}
\begin{cases}
- \Delta p = -\nabla \cdot f \ \ \ &\text{in $D$}\\
(\nabla p - f) \cdot n = 0 \ \ \ &\text{on $\partial D$.}
\end{cases},\ \ \ \int_D p = 0. 
\end{align}
Identity \eqref{weak.convergence.stokes} also implies that the full $\{u_\eps\}_{\eps >0}$ converges to $u^*$. 

\smallskip

As for the proof of Theorem \ref{t.main}, case (a), we restrict to the converging subsequence $\{ u_{\eps_j} \}_{j\in \N}$ but we skip the index $j\in \N$ in the notation. We start by noting that, using the divergence-free condition for $u_\eps$ and that $u_\eps$ vanishes on $\partial D$, we have that for every $\phi \in C^\infty(D)$ and $\chi \in L^\infty (\Omega)$
\begin{align}\label{first.condition.u}
\E \bigl[ \chi \int_D \nabla \phi \cdot u^* \bigr] = 0.
\end{align}

\smallskip

Let $\chi \in L^\infty(\Omega)$ and $\phi \in C^\infty_0(D)$ with $\nabla \cdot \phi =0$ in $D$ be fixed. For every $\delta > 0$, we appeal to Lemma \ref{l.reduction.2} to infer that there exists an $\eps_\delta>0$ and a set $A_\delta \in \mathcal{F}$, having  $\P(A_\delta) \geq 1-\delta$, such that for every $\omega \in A_\delta$ and for every $\eps \leq \eps_\delta$ we may consider the function $R_\eps \phi \in H^1_0(D^\eps)$ of Lemma \ref{l.reduction.2}. Testing equation \eqref{P.eps.s} with $R_\eps(\rho)$, and using that the vector field $R_\eps v$ is divergence-free, we infer that
\begin{align}
\E \bigl[ \1_{A_\delta} \chi \int_D \nabla u_\eps \colon \nabla (R_\eps \phi)  \bigr] = \E \bigl[ \chi \1_{A_\delta} \int_D (R_\eps \phi) f \bigr].
\end{align}
Using Lemma \ref{l.reduction.2} and the bounds of Lemma \ref{l.unif.bounds} this implies that in the limit $\eps \downarrow 0$ we have
\begin{align}
\E \bigl[  \1_{A_\delta} \chi \int_D (u^* - K f) \phi  \bigr] = 0.
\end{align}
We now send $\delta \downarrow 0$ and appeal to the Dominated Convergence Theorem to infer that
\begin{align}\label{condition.u.2}
\E \bigl[ \chi \int_D (u^* - K f) v  \bigr] = 0.
\end{align}

\smallskip

Since $D$ has $C^{1,1}$-boundary and is simply connected, the spaces $L^p(D)$, $p \in (1, +\infty)$ admit an $L^p$-Helmoltz decomposition $L^p(D)= L^p_{\text{div}}(D) \oplus  L^p_{\text{curl}}(D)$ \cite{Galdi_book}[Section III.1]. This, the separability of $L^p(D)$, $p \in [1, +\infty)$, and the arbitrariness of $\chi$ and $\phi$ in \eqref{condition.u.2},  allows us to infer that for $\P$-almost realization the function $u^*$ satisfies $u^* = K f + \nabla p(\omega; \cdot)$ for $p(\omega; \cdot ) \in W^{1,p}(D)$, $p \in [1; 2)$. By a similar argument, we may use \eqref{first.condition.u} to infer that for $\P$-almost every realization and for every $v\in W^{1.q}(D)$, $q > 2$ we have
$$
\int_D (\nabla p(\cdot; \omega) + Kf) \cdot \nabla v = 0.
$$
Since \eqref{eq.P} admits a unique mean-zero solution, we conclude that $p(\omega, \cdot)$ does not depend on $\omega$. Finally, since $D$ is regular enough and $f \in L^q(D)$, standard elliptic regularity yields that $p \in H^1(D)$. This concludes the proof of \eqref{weak.convergence.stokes}.

\smallskip

We now upgrade the convergence of the family $\{ u_\eps \}_{\eps >0}$ to $u^*$ from weak to strong: We claim that for every $\delta>0$ we may find a set $A_\delta \subset \Omega$ with $\P(A_\delta) > 1-\delta$ such that
\begin{align}\label{strong.b}
\lim_{\eps \downarrow 0} \E \bigl[ \1_{A_\delta} \int_D | \sigma_\eps^2u_\eps - u^* |^q \bigr] = 0.
\end{align}
Here, $q \in [1, 2)$. The proof of this inequality follows the same lines of the proof for \eqref{strong.a} in case $(a)$: In this case, we rely on Lemma \ref{l.reduction.2} instead of Lemma \ref{l.oscillating.p} and use that, thanks to the definition \eqref{eq.P}, it holds
\begin{align}
\int_D f (f - \nabla p ) = \int_D (f- \nabla p)^2.
\end{align}

\smallskip

From \eqref{strong.b}, the statement of Theorem \ref{t.main}, $(b)$ easily follows: Let, indeed, $q \in [1, 2)$ be fixed. For every $\delta >0$, let $A_\delta$ be as above. We rewrite
\begin{align}
\E \bigl[ \int |\sigma_\eps u_\eps - u |^q \bigr] = \E \bigl[ \1_{A_\delta}\int |\sigma_\eps u_\eps - u |^q \bigr] + \E \bigl[ \1_{\Omega \backslash A_\delta}\int |\sigma_\eps u_\eps - u |^q \bigr]
\end{align}
and, given an exponent $p \in (q; 2)$, we use H\"older's inequality and the assumption on $A_\delta$ to control
\begin{align}
\E \bigl[ \int |\sigma_\eps u_\eps - u |^q \bigr] \leq \E \bigl[ \1_{A_\delta}\int |\sigma_\eps u_\eps - u |^q \bigr] + \delta^{1- \frac{q}{p}}  \E \bigl[ \int |\sigma_\eps u_\eps - u |^{p} \bigr]^{\frac{q}{p}}.
\end{align}
Since by Lemma \ref{l.unif.bounds} the family $\sigma_\eps^2 u_\eps$ is uniformly bounded in every $L^p(\Omega \times D)$ for $p \in [1, 2)$, we establish
\begin{align}
\limsup_{\eps \downarrow 0}\E \bigl[ \int |\sigma_\eps u_\eps - u |^q \bigr] \leq \limsup_{\eps \downarrow 0}\E \bigl[ \1_{A_\delta}\int |\sigma_\eps u_\eps - u |^q \bigr] + \delta^{1- \frac{q}{p}} \stackrel{\eqref{strong.b}}{\lesssim} \delta^{1-\frac{q}{p}}.
\end{align}
Since $\delta$ is arbitrary, we conclude the proof of Theorem \ref{t.main}, $(b)$.
\end{proof}

\smallskip

\subsection{Proof of Lemma \ref{l.reduction.2}}
This section is devoted to arguing  Lemma \ref{l.reduction.2} by leveraging on the geometric information on the clusters of holes $H^\eps$ contained Lemma \ref{l.borel.cantelli}. The idea behind these proof is in spirit very similar to the one for Lemma \ref{l.oscillating.p} in case (a): As in that setting, indeed, we aim at partitioning the holes of $H^\eps$ into a subset  $H^\eps_g$ of disjoint and ``small enough'' holes and  $H^\eps_b$ where the clustering occurs. 

\medskip

The main difference with case $(a)$, however, is due to the fact that we need to ensure that the so-called Stokes capacity of the set $H^\eps_b$, namely the vector 
\begin{align}\label{stokes.capacity}
(\text{St-Cap} (H^\eps_b))_i = \inf\biggl\{ \int |\nabla v |^2 \, \, \colon \, \, v \in C^\infty_0(\R^3; \R^3), \, \, \, \nabla \cdot v = 0 \, \, \text{in $\R^3$,} \, \, v \geq e_i \ \ \text{in $H^\eps_b$} \biggr\}, \ \ \ i= 1, 2, 3
\end{align}
vanishes in the limit $\eps \downarrow 0$. The divergence-free constraint implies that, in contrast with the harmonic capacity of case (a), the Stokes capacity is not subadditive. This yields that, if $H^\eps_b$ is constructed as in Lemma \ref{l.geometry.p}, then we cannot simply control its Stokes-capacity by the sum of the capacity of each ball of $H^\eps_b$. 

\medskip

We circumvent this issue by relying on the information on the length of the clusters given by Lemma \ref{l.borel.cantelli}. We do this by adopting the exact same strategy used to tackle the same issue in the case of the Brinkmann scaling in \cite{GH}. The following result is a simple generalization of \cite{GH}[Lemma 3.2] and upgrades the partition of Lemma \ref{l.geometry.p} in such a way that we may control the Stokes-capacity of the clustering holes in $H^\eps_b$. For a detailed discussion on the main ideas behind this construction, we refer to \cite{GH}[Subsection 2.3].

\begin{lem} \label{l.geometry.s}
Let $\gamma > 0$ be as chosen in Lemma \ref{conv.measure}. For every $\delta > 0$ there exists $\eps_0 > 0$ and $A_\delta \subset \Omega$ with $\P(A_\delta) > 1-\delta$ such that for every $\omega \in \Omega$ and $\eps \leq \eps_0$ we may choose $H^\eps_g, H^\eps_b$ of Lemma \ref{l.geometry.p} as follows:
 \begin{itemize}
 \item There exist $ \Lambda(\beta)> 0$, a sub-collection $J^\eps \subset \mathcal I^\eps$ and constants $\{ \lambda_l^\eps \}_{z_l\in J^\eps} \subset [1, \Lambda]$ such that
	\begin{align}
		\label{bar.H^b}
	 H_b^\eps \subset \bar H^\eps_b := \bigcup_{z_j \in J^\eps} B_{\lambda_j^\eps \aeps \rho_j}( \eps z_j), \ \ \ \lambda_j^\eps \aeps \rho_j \leq \Lambda \eps^{\kappa}.
	\end{align}
 \item There exists $k_{max}= k_{max}(\beta, d)>0$ such that we may partition 
 	$$
	\mathcal I^\eps= \bigcup_{k=-3}^{k_{max}} \mathcal I_k^\eps, \ \ \ J^\eps= \bigcup_{i=-3}^{k_{max}} J_k^\eps,
	$$
 with $\I^\eps_k \subset J^\eps_k$ for all $k= 1, \cdots, \km$ and
{ \begin{align}\label{inclusion.step.by.step}
	\bigcup_{z_i \in \mathcal I_k^\eps} B_{\aeps \rho_i}( \eps z_i) \subset \bigcup_{z_j \in J_k^\eps} B_{\lambda_j^\eps \aeps \rho_j}( \eps z_j);
\end{align}}
 \item  For all $k=-3, \cdots, k_{max}$ and every $z_i, z_j \in J_k^\eps$, $z_i \neq z_j$
\begin{align}\label{similar.size.apart}
B_{\theta^2 \lambda_i^\eps \aeps \rho_i}(\eps z_i) \cap B_{\theta^2 \lambda_j^\eps \aeps \rho_j}(\eps z_j) = \emptyset;
\end{align}
\item For each $k=-3, \cdots, k_{max}$ and $z_i \in \mathcal I_k^\eps$ and for all $ z_j \in \bigcup_{l=-3}^{k-1} J_l^\eps$ we have
\begin{align}
\label{small.dont.intersect.big}
B_{\aeps \rho_i}(\eps z_i) \cap B_{\theta \lambda_j^\eps \aeps \rho_j}(\eps z_j) = \emptyset.
\end{align}
 \end{itemize}
 \smallskip
 Finally, the set $D^\eps_b$ of Lemma \ref{l.geometry.p} may be chosen as
\begin{align}
\label{D_b}
& D^\eps_b = \bigcup_{z_i \in J^\eps} B_{\theta \aeps \lambda_i^\eps\rho_i}(\eps z_i).
\end{align}

\smallskip

The same statement is true for $\P$-almost every $\omega \in \Omega$ for every $\eps > \eps_0$ (with $\eps_0$ depending, in this case, also on the realization $\omega$).
\end{lem}

\medskip

\begin{proof}[Proof of Lemma \ref{l.geometry.s}]
The proof of this result follows the exact same lines of  of \cite{GH}[Lemma 3.2]. We thus refer to it for the proof and to \cite{GH}[Subsection 3.1] for a sketch of the ideas behind the quite technical argument. We stress that the different scaling of the radii does not affect the argument since the necessary requirement is that $\eps^\alpha << \eps$. This holds for every choice of $\alpha \in (1, 3)$. We also emphasize that in the current setting, Lemma \ref{l.borel.cantelli} plays the role of \cite{GH}[Lemma 5.1]. This result is crucial as it provides information on the length of the overlapping balls of $H^\eps$. For every $\delta >0$, we thus select the set $A_\delta$ of Lemma \ref{l.borel.cantelli} containing those realizations where the partition of $H^\eps$ satisfies \eqref{max.radii} and \eqref{no.overlapping.borel}. Once restricted to the set $A_\delta$, the construction of the set $H^\eps_b$ is as in \cite{GH}[Lemma 3.1].
\end{proof}

\medskip

Equipped with the previous result, we may now proceed to prove Lemma \ref{l.reduction.2}:

\begin{proof}[Proof of Lemma \ref{l.reduction.2}]
The proof of this is similar to the one in \cite{GH}[Lemma 2.5] for the analogous operator and we sketch below the main steps and the main differences in the argument. For $\delta >0$, let $\eps_0>0$ and $A_\delta \subset \Omega$ be the set of Lemma \ref{l.geometry.s}; From now on, we restrict to the realization $\omega \in A_\delta$.  For every $\eps < \eps_0$ we appeal to Lemma \ref{l.geometry.p} and Lemma \ref{l.geometry.s} to partition $H^\eps= H^\eps_b \cup H^\eps_g$. We recall the definitions of the set $n^\eps \subset \Phi^\eps(D)$ in \eqref{good.set.ppp} in Lemma \ref{l.geometry.p} and of the subdomain $D^\eps_b \subset D$ in \eqref{D_b} of Lemma \ref{l.geometry.s}.

\smallskip

\textit{Step 1. (Construction of $R_\eps$)} For every $\phi \in C^\infty_0(D)$, we define $R_\eps \phi$ as 
\begin{align}\label{definition.Rv}
R_\eps \phi := \begin{cases}
\phi^\eps_b \ \ \ \text{ \ in $D^\eps_b$}\\
\phi^\eps_g \ \ \ \ \text{ in $D \backslash D^\eps_b$,}
\end{cases}
\end{align}
where the functions $\phi^\eps_b$ and $\phi^\eps_g$ satisfy
\begin{align}\label{Rv.bad}
\begin{cases}
\phi^\eps_{b} = 0 \ \text{ in $H^\eps_b$}, \ \ \ \ \phi^\eps_{b}= \phi \ \text{ in $D \backslash D^\eps_b$,}\vspace{0.2cm}\\
\nabla \cdot \phi^\eps_b = 0 \ \text{ in $D$}, \vspace{0.2cm}\\
 \| \phi^\eps_{b} - \phi \|_{L^p}^p \lesssim_p |D^\eps_b| \ \ \text{for every $p\geq 1$,}\vspace{0.2cm}\\
\| \sigma_\eps \nabla \phi_\eps^b \|_{L^2}^2 \lesssim \eps^3 \sum_{z \in \Phi^\eps(D) \backslash n^\eps} \rho_z.\\
\end{cases}
\end{align}
and
\begin{align}\label{Rv.good}
\begin{cases}
  \phi^\eps_g= \phi \ \text{ in $D^\eps_b$}, \ \ \ \ \phi^\eps_g = 0 \ \text{in $H^\eps_g$,}\vspace{0.2cm}\\
  \nabla \cdot \phi^\eps_g =0 \ \ \ \text{in $D$,}\vspace{0.2cm}\\
\| \nabla (\phi^\eps_g - \phi) \|_{L^2(D)}^2 \lesssim \eps^\alpha \sum_{z \in n^\eps(D)} \rho_z,\vspace{0.2cm}\\
\| \phi^\eps_g - \phi \|_{L^p(D)}^p \lesssim \eps^{3 \delta + 3} \sum_{z\in n^\eps(D)}\rho_z^{\frac 3 \alpha + \beta}.
\end{cases}
\end{align}

\bigskip

\textit{ Step 2. (Construction of $\phi^b_\eps$)}We construct $\phi^\eps_b$ as done in \cite{GH}[Proof of Lemma 2.5, Step 2]: For every $z \in J^\eps$, we define
\begin{align}
B_{\theta, z}:= B_{\theta \lambda_\eps \aeps \rho_z}(\eps z), \ \ \ B_{z}:=B_{\lambda_\eps \aeps \rho_z}(\eps z).
\end{align}
It is clear that the previous quantities also depend on $\eps$. However, in order to keep a leaner notation, we skip it in the notation. We use the same understanding for the function $\phi^\eps_b$ and the sets $\{ I_{\eps,i}\}_{i=-3}^{\km}$ and $\{ J_{\eps,i} \}_{i=-3}^{\km}$ of Lemma \ref{l.geometry.s}.

\smallskip

We define $\phi^b$ by solving a finite number of boundary value problems in the annuli 
$$
\bigcup_{z \in I_{k}} B_{\theta,z} \backslash B_z, \ \ \ \text{for $k=-3, \cdots, \km$ }
$$
We stress  that, thanks to Lemma \ref{l.geometry.s}, for every $k= -3, \cdots, \km$, each one of the above collections contains only disjoint annuli. Let $\phi^{(\km +1)}= \phi$. Starting from $k= \km$, at every iteration step $k= \km, \cdots, -3$, we solve for every $z \in I_{\eps,k}$ the Stokes system
\begin{align}
\begin{cases}
-\Delta \phi^{(k)} + \nabla \pi^{(k)} = -\Delta \phi^{(k+1)} \ \ \ &\text{in $B_{\theta,z} \backslash B_z$}\\
\nabla \cdot \phi^{(k)}=0 \ \ \ \ &\text{in $B_{\theta,z}\backslash B_z$}\\
\phi^{(k)}= 0 \ \ \ \ &\text{on $\partial B_{\theta,z}$}\\
\phi^{(k)}= \phi^{(k+1)} \ \ \ &\text{on $\partial B_{z}$.}
\end{cases}
\end{align}
We then extend $\phi^{(k)}$ to $\phi^{(k+1)}$ outside  $\bigcup_{z \in I_k} B_{\theta,z}$ and to zero in $\bigcup_{z \in I_{k}} B_{z}$.

\smallskip

The analogue of inequalities of \cite{GH}[(4.12)-(4.14)], this time with the factor $\eps^{\frac{d-2}{d}}$ replaced by $\eps^\alpha$ and with $d=3$, is
 \begin{equation}
 \label{iteration.estimate}
\begin{aligned}
&\| \nabla \phi^{(k)} \|_{L^2(D)}^2 \lesssim \| \nabla \phi \|_{L^2(D)}^2 + \eps^d \sum_{z \in \cup_{i=k}^{\km}J_{i} } \rho_z  \| \phi \|_{L^\infty(D)}^2,\\
&\| \phi^{(k)} \|_{C^0(D)} \lesssim \| \phi\|_{C^0({D})},
\end{aligned}
\end{equation}
and 
\begin{align}\label{vanishing.set}
\nabla \cdot \phi^{(k)} = 0 \, \, \text{in $D$}, \ \ \ \phi^{i} = 0 \ \ \ \ \text{ in }  \bigcup_{z \in \bigcup_{i=k}^{\km} \I_{i}} B_{\aeps \rho_z}(\eps z).
\end{align}
Moreover,
\begin{equation}
\label{strong.convergence}
\begin{aligned}
& \phi^{(k)} - \phi = 0  \ \ \ \ \text{ in } D \backslash \left(   \bigcup_{z \in \cup_{i=k}^{\km} J_{i}} B_{\theta,z} \right),\\
&\| \nabla (\phi^{(k)} -\phi) \|_{L^2(D)}^2 \lesssim  \sum_{z \in \cup_{i=k}^{\km}J_{i} } \!\!\!\!\! \!\!\!\Bigl( \| \nabla \phi \|_{L^2(B_{\theta,z})}^2 + \eps^d \rho_z  \| \phi \|_{L^\infty(D)}^2\Bigr).
\end{aligned}
\end{equation}
These inequalities may be proven exactly as in \cite{GH}.  We stress that condition \eqref{small.dont.intersect.big} in Lemma \ref{l.geometry.s} is crucial in order to ensure that this construction satisfies the right-boundary conditions. In other words, the main role of Lemma \ref{l.geometry.s} is to ensure that, if at step $k$ the function $\phi^{(k)}$ vanishes on a certain subset of $H^\eps_b$, then $\phi^{(k+1)}$ also vanishes in that set (and actually vanishes on a bigger set).

\smallskip

We set $\phi^\eps_b = \phi^{(-3)}$ obtained by the previous iteration. The first property in \eqref{Rv.bad} is an easy consequence of \eqref{vanishing.set} and the first identity in \eqref{strong.convergence}. We recall, indeed, that thanks to Lemma \ref{l.geometry.s} we have that
 $$
 H^\eps_b =  \bigcup_{z \in \bigcup_{i=-3}^{\km} \I_{i}} B_{\aeps \rho_z}(\eps z), \ \  D^\eps_b= \bigcup_{z \in \cup_{k=-3}^{\km} J_{i}} B_{\theta,z}.
 $$
The second property in \eqref{Rv.bad} follows immediately from \eqref{vanishing.set}. The third line in \eqref{Rv.bad} is an easy consequence of the first line in \eqref{Rv.bad} and the second inequality in \eqref{iteration.estimate}. Finally, the last inequality in \eqref{Rv.bad} follows by multiplying the last inequality in \eqref{strong.convergence} with the factor $\sigma_\eps$ and using that, since $\phi \in C^\infty$, we have that
\begin{align}
\|\sigma_\eps^{-1}\nabla (\phi^\eps_b-\phi) \|_{L^2(D)}^2 \lesssim \|\phi\|_{C^1(D)} \eps^{3-\alpha} \sum_{z \in \cup_{k=-3}^{\km}J_{k} }( \eps^{3\alpha} \rho_z^3 +\aeps \rho_z )\lesssim \eps^3 \sum_{z \in \cup_{k=-3}^{\km}J_{k} }((\eps^{\alpha} \rho_z)^2 + 1)\rho_z.
\end{align}
Thanks to Lemma \ref{l.geometry.s} and the definition of the set $n^\eps$ in Lemma \ref{l.geometry.p}, the previous inequality yields the last bound in \eqref{Rv.bad}.

\bigskip

\textit{Step 3. (Construction of $\phi^\eps_g$)}\, Equipped with $\phi^\eps_b$ satisfying \eqref{Rv.bad}, we now turn to the construction of $\phi_g^\eps$. Also in this case, we follow the same lines of \cite{GH}[Proof of Lemma 2.5, Step 3] and exploit the fact that the set $H^\eps_g$ is only made by balls that are disjoint and have radii $\aeps\rho$ that are sufficiently small. We define the function $\phi_g^\eps$ exactly as in \cite{GH}[Proof of Lemma 2.5, Step 3] with the radius $a_{i,\eps}$ in \cite{GH}[(4.18)] being defined as  $a_{\eps,z}= \aeps\rho_z$ instead of $\eps^{\frac{d-2}{d}}\rho_z$. More precisely, for every $z \in n^\eps$, we write
\begin{align}\label{abbreviations}
a_{\eps,z}:= \aeps \rho_z, \ \ \ \ \ d_{\eps,z}:= \min \biggl\{ \operatorname{dist}(\eps z, D^\eps_b), \frac 1 2 \min_{\tilde z \in n^\eps, \atop z \neq \tilde z} \bigl( \eps | z - \tilde z| \bigr), \eps \biggr\}
\end{align}
and we set
\begin{align}\label{notation.Allaire}
T_z = B_{a_{\eps,z}} (\eps z), \ \ B_z:= B_{\frac {d_{\eps,z}}{ 2}} (\eps z), \ \ B_{2,i}:= B_{d_{\eps,z}}(\eps z), \ \ C_z:= B_z \backslash T_z, \ \ D_z:= B_{2,z} \backslash B_z.
\end{align}
With this notation, we define the function $\phi^\eps_g$ as in \cite{GH}[(4.19)-(4-21)]. Also in this case, identities, \cite{GH}[(4.22)-(4.23)] hold. By Lemma \ref{l.geometry.p}  It is immediate to see that this construction satisfies the first two properties in \eqref{Rv.good}.

\smallskip

We now turn to show the remaining part of \eqref{Rv.good}: We remark that, since $z \in n^\eps(D)$, Lemma \ref{l.geometry.p} and definition \eqref{abbreviations} yield that
\begin{align}\label{bounds.a.z}
(\frac{a_{\eps,z}}{d_{\eps,z}}) \leq \eps^{\frac{\gamma}{2}},\ \ \ a_{\eps,z}^3 \leq \eps^{3+ 3\gamma} \rho_{z}^{\frac 3 \alpha + \beta},
\end{align}
where $\gamma>0$ is as in Lemma \ref{l.geometry.s} and $\beta>0$ is as in \eqref{integrability.s}.
Equipped with the previous bounds, the analogue of estimates \cite{GH}[(4.26)-(4.30)] yield that for every $z \in n^\eps(D)$
\begin{align}\label{inequalities.phi.good}
\| \nabla (\phi^\eps_g - \phi)\|_{L^2(D_i)}^2 \lesssim \eps^\gamma \eps^\alpha \rho_z, \ \ \ \ \ \  \|  \phi^\eps_g - \phi \|_{L^p(D_i)}^p \lesssim \eps^{\gamma p} d_{\eps,z}^3\\
 \| \nabla (\phi^\eps_g - \phi)\|_{L^2(C_i)}^2 \lesssim \eps^\alpha \rho_z,\ \ \ \ \ \ \|\phi^\eps_g - \phi\|_{L^p(C_i)}^p \lesssim \eps^{3\gamma + 3}\rho_z^{\frac 3 \alpha + \beta}\\
\| \nabla (\phi^\eps_g - \phi)\|_{L^2(T_i)}^2 + \|\phi^\eps_g - \phi\|_{L^p(T_i)}^p  \lesssim \eps^{3\gamma + 3}\rho_z^{\frac 3 \alpha + \beta}.
 \end{align}
Since $B_{2,z} = D_z \cup C_z\cup T_z$ and the function $\phi^\eps_g - \phi$ is supported only on $\bigcup_{z\in n^\eps(D)}B_{2,z}$, we infer that
for every $z \in n^\eps(D)$, it holds 
\begin{align}
\| \nabla (\phi^\eps_g - \phi) \|_{C^0(B_z)} \lesssim \eps^\alpha\rho_z +  \eps^{3 \gamma + 3}\rho_z^{\frac 3 \alpha + \beta}, \ \ \ 
\| \phi^\eps_g - \phi \|_{L^p(B_z)}^p \lesssim \eps^{3 \delta + 3} \rho_z^{\frac 3 \alpha + \beta}.
\end{align}
Summing over $z \in n^\eps$ we obtain the last two inequalities in \eqref{Rv.good}. We thus established \eqref{Rv.good} and completed the proof of Step 1.

\bigskip

{\bf Step 4. (Properties of $R_\eps$)}\, We now argue that $R_\eps$ defined in Step 1. satisfies all the properties enumerated in Lemma \ref{l.geometry.s}. It is immediate to see from \eqref{Rv.good} and \eqref{Rv.bad} that $R_\eps \phi$ vanishes on $H^\eps$ and is divergence-free in $D$. Inequalities \eqref{aver.R} also follow easily from the inequalities in \eqref{Rv.good} and \eqref{Rv.bad} and arguments analogous to the ones in Lemma \ref{l.oscillating.p}. We stress that, in this case, we appeal to condition \eqref{integrability.s} and, in the expectation, we need to restrict to the subset $A_\delta \subset \Omega$ of the realizations for which $R_\eps$ may be constructed as in Step 1.

\smallskip

To conclude the proof, it only remains to tackle \eqref{meas.R}. We do this by relying on the same ideas used in Lemma \ref{l.oscillating.p} in the case of the Poisson equation. We use the same notation introduced in Step 2. We begin by claiming that \eqref{meas.R} reduces to show that for every $i=1, \cdots , 3$
\begin{align}\label{meas.R.3}
\lim_{\eps \downarrow 0}\E \bigl[ |\sum_{z \in n^\eps(D)} \int_{\partial B_z}(\partial_\nu w_{\eps,z}^i - q_{\eps,z}^i \nu_i) \phi_i v_{\eps,i} - K^{-1}\int_D v_{\eps, i} \phi_i | \bigr]  = 0,
\end{align}
where
 $$
w_{\eps,z}^i(x) := \bar w_{i}(\frac{x - \eps z}{\aeps \rho_z}), \ \ \ q_{\eps,z}^i(x)= (\aeps \rho_z)^{-1} \bar q_i(\frac{x - \eps z}{\aeps \rho_z}), \ \ \ x\in {B_z},
$$
with $(\bar w_i , \bar q_i)$ solving 
\begin{align}\label{cell.stokes}
\begin{cases}
\Delta \bar w_i- \nabla \bar q_i = 0 \ &\text{in $\Rd \backslash B_1$}\\
\nabla \cdot \bar w_i  = 0 \ &\text{in $\Rd \backslash B_1$}\\
\bar w_i = e_i  \ &\text{on $\partial B_1$}\\
\bar w_i \to 0 \ \ \ &\text{for $|x| \to +\infty$}.
\end{cases}
\end{align}

\smallskip

We use the definition of $R_\eps \phi$ to rewrite for every $\omega \in \1_{A_\delta}$
\begin{align}\label{rewrite.term}
 \int_D \nabla v_\eps \cdot \nabla R_\eps(\phi)&  =  \int_{D} \nabla v_\eps \cdot \nabla( \phi^\eps_g - \phi) + \int_{D} \nabla v_\eps \cdot \nabla (\phi^\eps_b - \phi) + \int_D \nabla v_\eps \cdot \nabla \phi.
 \end{align}
We claim that, after multiplying by $\1_{A_\delta}$ and taking the expectation, the last two integrals on the right-hand side vanish in the limit.
In fact, using the triangle and Cauchy-Schwarz's inequalities and combining them with \eqref{Rv.bad} and the uniform bounds for $\{v_\eps \}_{\eps>0}$ we have that
\begin{align}\label{first.term}
\limsup_{\eps \downarrow 0}\E \bigl[\1_{A_\delta} |\int_{D} \nabla v_\eps \cdot \nabla( \phi^\eps_b - \phi) + \int_D \nabla v_\eps \cdot \nabla \phi|\bigr] \lesssim
\limsup_{\downarrow 0}\E \bigl[\1_{A_\delta} \eps^3 \sum_{z \in \Phi^\eps(D) \backslash n^\eps(D)} \rho_z \bigr]^{\frac 1 2}\stackrel{\eqref{bad.cap.vanishes}}{=}0.
\end{align}

Hence, we show \eqref{meas.R} provided that
\begin{align}
\lim_{\eps \downarrow 0}\E \bigl[ \1_{A_\delta}|  \int_{D} \nabla v_\eps \cdot \nabla( \phi^\eps_g - \phi) - K^{-1}\int_D v \cdot \phi | \bigr]  = 0.
\end{align}
Furthermore, since $\sigma_\eps^{-2}v_\eps \rightharpoonup v$ in $L^p(\Omega \times D)$, $p \in [1, 2)$ and $\phi \in C^\infty_0(D)$, it suffices to prove that
\begin{align}
\lim_{\eps \downarrow 0}\E \bigl[ \1_{A_\delta}|  \int_{D} \nabla v_\eps \cdot \nabla( \phi^\eps_g - \phi) -K^{-1}\int_D \sigma_\eps^{-2}v_\eps \cdot  \phi | \bigr]  = 0.
\end{align}
We further reduce this to \eqref{meas.R.3} if
\begin{align}\label{meas.R.6}
\lim_{\eps\downarrow 0} \E\bigl[  \1_{A_\delta}|  \int_{D} \nabla v_\eps \cdot \nabla( \phi^\eps_g - \phi) - \sum_{z \in n^\eps(D)} \int_{\partial B_z}(\partial_\nu w_{\eps,z}^i - q_{\eps,z}^i \nu_i) \phi_i v_{\eps,i}|  \bigr] =0.
\end{align}
An argument analogous to the one outlined in \cite{GH} to pass from the left-hand side of \cite{GH}[(4.34)] to the one in \cite{GH}[(4.39)] yields that
\begin{align}\label{meas.R.5}
\lim_{\eps\downarrow 0} \E\bigl[  \1_{A_\delta}|\int_{D} \nabla v_\eps \cdot \nabla( \phi^\eps_g - \phi) - \sum_{z \in n^\eps(D)} \phi_i(\eps z)\int_{\partial B_z}(\partial_\nu w_{\eps,z}^i - q_{\eps,z}^i \nu_i)  v_{\eps,i}|  \bigr] =0.
\end{align}
We stress that in the current setting we use again the uniform bounds on the sequence $\sigma_\eps^{-1}\nabla u_\eps$ and we rely on estimates \eqref{inequalities.phi.good} instead of \cite{GH}[(4.26)-(4.30)]. To pass from \eqref{meas.R.5} to \eqref{meas.R.3} it suffices to use the smoothness of $\phi$ and, again, the bounds on the family $\{v_\eps\}_{\eps>0}$. We thus established that \eqref{meas.R} reduces to \eqref{meas.R.3}.

\bigskip

We finally turn to the proof of \eqref{meas.R.3}. By the triangle inequality it suffices to show that
\begin{align}\label{meas.R.4}
\lim_{\eps \downarrow 0}\E \bigl[ |\langle \, \tilde \mu_{\eps,i} ; \phi_i v_{\eps,i} \, \rangle -  K^{-1} \int v_{\eps,i} \phi_i | \bigr]  = 0 \ \ \ \text{for all $i=1, 2 , 3$}
\end{align}
where the measures $\mu_{\eps,i} \in H^{-1}(D)$, $i=1, 2 ,3$, are defined as
\begin{align}\label{mu.tilde}
\tilde \mu_{\eps,i} := \sum_{z \in n^\eps(D)} g_{\eps,z}^i \delta_{\partial B_z}, \ \ \ g_{\eps,z}^i := (\partial_\nu w_{\eps,z}^i - q_{\eps,z}^i \nu_i ).
\end{align}
We focus on the limit above in the case $i=1$. The other values of $i$ follow analogously. We skip the index $i=1$ in all the previous objects. As done in the proof of  \eqref{conv.Delta} in Lemma \ref{l.oscillating.p}, it suffices to show that there exists a positive exponent $\kappa >0$ such that
\begin{align}\label{Delta.meas.stokes}
\E \bigl[ |\langle \tilde\mu_{\eps} ; \phi v_\eps \rangle - K^{-1} \int v_\eps \phi | \bigr] \leq \eps^\kappa \biggl(\int_D |\sigma_\eps^{-1} \nabla v_\eps|^2 + \int_D |\sigma_\eps^{-2} v_\eps|^2 \biggr)^{\frac 12 } + r_\eps,
\end{align}
with $\lim_{\eps \downarrow 0} r_\eps= 0$. From this, \eqref{meas.R.4} follows immediately thanks to the bounds assumed for $\{v_\eps\}_{\eps>0}$.

\smallskip

The proof of \eqref{Delta.meas.stokes} is similar to \eqref{conv.Delta}: For $k\in \N$ to be fixed, we apply once Lemma \ref{Kohn_Vogelius.general} to this new measure $\sigma_{\eps}^{-2}\tilde \mu_{\eps}$, with $\mathcal{Z}= \{\eps z\}_{z \in n^\eps(D)}$, $\mathcal{R} = \{ d_{\eps,z}\}_{z\in \tilde\Phi^\eps(D)}$, $\{g_{z,\eps}\}_{z \in \tilde\Psi^\eps(D)}$ and with the partition $\{ K_{\eps,z,k}\}_{z\in N_{k,\eps}}$ constructed in Step 1 in the proof of Lemma \ref{conv.measure}. This implies that 
\begin{align}\label{measures.mu.i}
\|\sigma_{\eps}^{-2}\tilde \mu_\eps - \tilde\mu_\eps(k) \|_{H^{-1}} &\lesssim k\eps \bigl(\sigma_{\eps}^{-2}\sum_{z \in \tilde \Phi^\eps(D)} d_{\eps,z}^{-1} \int_{\partial B_z} |g_{z,\eps}|^2  \bigr)^{\frac 1 2}\\
\tilde\mu_\eps(k) &:= \sum_{x \in N_{\eps,k}} \bigl(\frac{1}{|K_{\eps,x,k}|}\sum_{z \in N_{k,x,\eps}}\sigma_{\eps}^{-2} \int_{\partial B_z} g_{\eps,z} \bigr)\1_{K_{\eps,k,x}}
\end{align}

\smallskip

Appealing to the definition of $g_{\eps,z}$ and to the bounds for $(\bar w, \bar q)$ obtained in \cite{Allaire_arma1}[Appendix], for each $z \in n^\eps(D)$ it holds that
\begin{align}\label{bounds.g.i}
\sigma_\eps^{-2}\int_{\partial B_z} |g_{\eps, z}|^2 \lesssim \eps^3 \rho_z^2 d_z^{-2}, \ \ \ |\sigma_\eps^{-2}\int_{\partial B_z} g_{\eps,z} - 6\pi  \eps^3\rho_z| \lesssim \eps^3\rho_z (\frac{\eps^\alpha \rho_z}{\eps d_z}) \stackrel{\eqref{bounds.a.z}}{\lesssim} \eps^{3+ \frac \gamma 2}.
\end{align}
This, \eqref{measures.mu.i}, \eqref{mu.tilde}, the triangle inequality and the definition of $K^{-1}$, imply that
\begin{align}
\E \bigl[ |\langle \tilde\mu_{\eps} ; \phi v_\eps \rangle - K^{-1} \int v_\eps \phi | \bigr] &\lesssim  \bigl(\eps^3 \sum_{z \in n^\eps(D)} \rho_z^2 d_z^{-3}\bigr)^{\frac 1 2}  \bigl(\int_D |\nabla(\phi v_\eps)|^2\bigr)^{\frac 1 2}\\
&\quad\quad  + \bigl( \int_D| \frac{6}{4}m_\eps(k) - k^{-1}|^2 \bigr)^{\frac 12 } \bigl(\int_D |\sigma_\eps^{-2} v_\eps|^2 \bigr)^{\frac 12 } + \eps^\gamma,
\end{align}
where $\mu_\eps(k)$ is as in Step 2 of Lemma \ref{conv.measure} and $k$ is as in Theorem \ref{t.main}, $(a)$. From this, we argue \eqref{meas.R} exactly as done in Step 2-5 of Lemma \ref{conv.measure}. We established Lemma \ref{l.reduction.2}. 
\end{proof}

\section{Appendix}\label{s.appendix}
\begin{proof}[Proof of Lemma \ref{zero.one.law}] 

$(i) \Rightarrow (ii)$: We prove that
\begin{align}\label{zero.volume}
\lim_{\eps \downarrow 0} |H^\eps \cap D| = 0 \ \ \ \ \text{$\P$-almost surely.}
\end{align}
We do this by bounding
\begin{align}
 |H^\eps \cap D| \leq \sum_{z \in \Phi^\eps(D)} (\aeps\rho \wedge 1)^3 \leq \eps^{3\alpha}\sum_{z \in \Phi^\eps(D)} \rho^3 \1_{\rho < \eps^{-\alpha}} + \sum_{z \in \Phi^\eps(D)}\1_{\rho > \eps^{-\alpha}}  \sum_{z \in \Phi^\eps(D)} \rho_z
 \end{align}
and, for $0 < \delta < \alpha -1$,
\begin{align}
 |H^\eps \cap D| &\leq \eps^{3+ 3\delta} \#(\Phi^\eps(D)) +  \eps^{3\alpha}\sum_{z \in \Phi^\eps(D)} \rho^3 \1_{\eps^{-(\alpha-1)+\delta} <  \rho < \eps^{-\alpha}} + \sum_{z \in \Phi^\eps(D)}\1_{\rho > \eps^{-\alpha}}\\
 &{\lesssim} \eps^{3+ 3\delta} \#(\Phi^\eps(D)) + \sum_{z \in \Phi^\eps(D)} \rho^{\frac 3 \alpha} \1_{\rho > \eps^{-(\alpha-1)+\delta}} + \eps^3 \sum_{z \in \Phi^\eps(D)} \rho^{\frac 3 \alpha }\1_{\rho > \eps^{-\alpha}} \\
 &\lesssim  \eps^{3+ 3\delta} \#(\Phi^\eps(D)) + \sum_{z \in \Phi^\eps(D)} \rho^{\frac 3 \alpha} \1_{\rho > \eps^{-(\alpha-1)+\delta}} .
\end{align}
Since $\Phi$ is a Poisson point process and we assumed \eqref{integrability.p}, the right-hand side above vanishes $\P$-almost surely in the limit $\eps \downarrow 0$. This concludes the proof of \eqref{zero.volume} and immediately yields $(ii)$.

\smallskip

$(ii) \Rightarrow (i)$: This is equivalent to show that if $\E\bigl[ \rho^{\frac 3 \alpha} \bigr]= +\infty$ then for $\P$-almost every realization and $\eps$ small enough, the set $D^\eps =\emptyset$. With no loss of generality, let us assume that diam$(D) = 1$. We claim that if $\eps_j:= 2^{-j}$, then the events
\begin{align}
A_j:= \biggl\{ B_2(0) \subset B_{\frac 1 4 \eps_j^{\alpha} \rho_z} (\eps_j z) \, \, \text{for some $z \in \Phi^{\eps_j} (D) \backslash \Phi^{\eps_{j-1}}(D)$} \biggr\}, \ \ \ j\in \N
\end{align}
satisfy 
\begin{align}\label{borel.cantelli.ii}
\sum_{j\in \N} \P(A_j) = + \infty.
\end{align}
Since the events are independent, by Borel-Cantelli's Lemma we conclude that for $\P$-almost every realization there exists $j_0 \in \N$ such that for all $j \geq j_0$ we have $B_{2}(0) \subset B_{\frac 1 4 \eps_j^\alpha \rho_z}(\eps_j \rho_z)$, for some $ z\in \Phi^{\eps_j}(D)$.

We now argue that this suffices to prove that, for $\P$-almost every realization and all $\eps < 2^{-j_0}$, with $j_0 \in \N$ as above, there is an element $z \in \Phi^\eps(D)$ such that $B_1(0) \subset B_{\aeps \rho_z}(\eps z)$. Let, indeed, assume that $\eps_{j+1} \leq \eps \leq \eps_j$. Then, since $\Phi^{\eps_j}(D) \subset \Phi^\eps(D)$, we may find  $z \in \Phi^\eps(D)$ such that  $B_{2}(0) \subset B_{\frac 1 4 \eps_j^\alpha \rho_z}(\eps_j \rho_z)$, i.e. $|\eps_j z | \leq \frac 1 4  \eps_j^\alpha \rho_z - 2$. This, in particular, yields that
$$
|\eps z | \leq \frac 1 4 (\frac{\eps_j}{\eps})^{\alpha-1} \eps^\alpha \rho_z - 2\frac{\eps}{\eps_j} \stackrel{\alpha < 3}{\leq} \eps^\alpha \rho_z - 1,
$$
i.e. $B_1(0)\subset B_{\eps^\alpha\rho_z}(\eps z)$ for $ z\in \Phi^\eps(D)$.

\bigskip

We argue \eqref{borel.cantelli.ii}:  Let
$$
B_j =  \{(z, \rho_z) \in (\frac{1}{\eps_{j}} D \backslash \frac{1}{\eps_{j-1}}D) \times \R_+ \, \colon \,  \eps_j |z | + 2 < \frac 1 4 \eps_j^\alpha \rho_z \}.
$$
then, if $\Psi= (\Phi; \mathcal{R})$ denotes the extended point process on $\R^d \times [1; +\infty)$ with intensity $\tilde \lambda (x, \rho) = \lambda f(\rho)$ (c.f. Section \ref{s.process}), we rewrite
\begin{align}
\P( A_j) = 1 - \P( \Psi( B_j ) = 0 )& = 1- \exp\biggl( -\lambda \int_{\frac{1}{\eps_{j}} D \backslash \frac{1}{\eps_{j-1}}D} \int_1^{+\infty} \1_{B_j}(x) f(\rho) \d\rho \d x   \biggr).
\end{align}
Since
\begin{align}
  \int_{\frac{1}{\eps_{j}} D \backslash \frac{1}{\eps_{j-1}}D} \int_1^{+\infty} \1_{B_j}(x) f(\rho) \d\rho \d x  &= \int_1^{+\infty} f(\rho) \1_{\rho > \eps_j^{-\alpha}} \int_{\frac{1}{\eps_{j}} D \backslash \frac{1}{\eps_{j-1}}D} \d x \\
 & \gtrsim |D|\eps_j^{-3}\int_1^{+\infty} \1_{\rho > 12\eps_j^{-\alpha}} \gtrsim \eps_j^{-3}\P(  12\eps_j^{-\alpha} < \rho < 24 \eps_j^{-\alpha}),
\end{align}
we bound
\begin{align}
\P( A_j)  \geq 1- \exp\biggl( C \eps_j^{-3} \P(  12\eps_j^{-\alpha} < \rho < 24 \eps_j^{-\alpha})\biggr).
\end{align}
Recalling that $\eps_j = 2^{-j}$, we may sum over $j \in \N$ in the previous inequality and get that
\begin{align}
\sum_{j\in \N}\P( A_j) \geq  \sum_{j\in \N} (1- \exp\bigl\{ - C \eps_j^{-3} \P(12 \eps_{j}^{-\alpha}\leq  \rho \leq 24\eps_{j+1}^{-\alpha}) \bigr\}) 
\end{align}
We may assume that $\eps_j^{-3} \P( 12\eps_{j}^{-\alpha}\leq  \rho \leq 24\eps_{j+1}^{-\alpha}) \to 0$. If not, indeed, \eqref{borel.cantelli.ii} immediately follows. Since $\eps_j = 2^{-j}$, we have that 
\begin{align}
\sum_{j\in \N}\P( A_j ) \gtrsim \sum_{j\in \N} \eps_j^{-3} \P( 12\eps_{j}^{-\alpha}\leq  \rho \leq 24\eps_{j+1}^{-\alpha}) \gtrsim \sum_{j\in \N} \E\bigl[ \rho^{\frac 3 \alpha} \1_{12 \eps_{j}^{-\alpha}\leq  \rho \leq 24\eps_{j+1}^{-\alpha}}  \bigr] \simeq \E \bigl[ \rho^{\frac 3 \alpha} \bigr].
\end{align}
By the assumption $\E\bigl[ \rho^{\frac 3 \alpha}\bigr]= + \infty$, this establishes \eqref{borel.cantelli.ii}. The proof of Lemma \ref{zero.one.law} is complete.
\end{proof}

\begin{proof}[Proof of Lemma \ref{l.borel.cantelli}]
The proof of this lemma relies on an application of Borel-Cantelli's lemma and follows the same lines of the one in \cite{GH}[Lemma 5.1]. 

\smallskip

For $\kappa> 0$, let $\km = \lfloor \frac 1 \kappa \rfloor +1$. We partition the set of centres $\Phi^\eps(D)$ in terms of magnitude of the associated radii: We write $\Phi^\eps(D)=\bigcup_{k=-3}^{k_{\textrm{max}}} I_{\eps, k}$ with
\begin{align}
I_{\eps, -3}:= \{ z \in \Phi^\eps(D) \, \colon \, \eps^\alpha \rho_z < \eps^{1+2\kappa} \}, \ \ \ \ \  I_{\eps, \km}:= \bigl\{ z \in \Phi^\eps(D) \, \colon \, \aeps \rho_z \geq \eps^{1- \km \kappa} \bigr\} \\
 I_{\eps, k}:=  \{ z \in \Phi^\eps(D) \, \colon \, \eps^{1- k \kappa } \leq \eps^\alpha \rho_z < \eps^{1-(k+1) \kappa} \} \ \ \  \text{for  $-2 \leq k \leq \km -1$.}
\end{align}
Note that,  up to a relabelling of the indices $k=-3, \cdots, \km$, the previous partition satisfies \eqref{partition.magnitude} of Lemma \ref{l.borel.cantelli}.

\smallskip

For any set $\chi \subset \Phi^\eps(D)$, we say that $A$ contains a chain of length $M\in \N$, $M \geq 2$, if there exist $z_1, \cdots, z_M \in \chi$ such that $B_{4\aeps \rho_{z_i}}(\eps z_i) \cap B_{4\aeps \rho_{z_j}}(\eps z_j)\neq \emptyset$, for all $i, j = 1, \cdots M$. We say that $A$ contains a chain of size $1$ if and only if $A \neq \emptyset$.

\smallskip

Equipped with this notation, \eqref{max.radii} follows provided we argue that for $\kappa$ suitably chosen, there exists  $k_0 < \km -1$ such that, $\P$-almost surely and for $\eps$ small, the sets $\{ I_{\eps,k} \cup I_{\eps, k+1} \}_{k=k_0}^{\km}$ are empty. This is equivalent to prove that they do not contain any chain of size at least $1$. Similarly, \eqref{no.overlapping.borel} is obtained if we find an $M \in \N$ such that $\P$-almost surely and for $\eps$ small enough, all the sets $\{ I_{\eps,k} \cup I_{\eps, k+1} \}_{k=-3}^{k_0}$ contain chains of length at most $M-1$.

\smallskip

For $M \in \N$ and $k= -3, \cdots, \km$, we define the events
$$
A_{k, \eps,M}:= \bigl\{  I_{\eps,k} \cup I_{\eps, k+1} \,  \text{contains a chain of length at least $M$}\bigr\}.
$$
We claim that if $\kappa < \min \bigl( \frac{\beta}{6} ; \frac{\alpha^2\beta}{6+ 2\alpha\beta} \bigr)$, then there exists $k_0 \in \N$, $k_0 < \km$ such that for every $k  \in \{ k_0, \cdots,  \km \}$
\begin{align}\label{limsup.borel.1}
\P( \bigcap_{\eps_0> 0} \bigcup_{\eps < \eps_0} A_{k, \eps,1}) =0
\end{align}
and there exists $M=M(\alpha, \beta) \in \N$ such that for every $k = -3, \cdots, k_0 -1$, also
\begin{align}\label{limsup.borel.2}
\P( \bigcap_{\eps_0> 0} \bigcup_{\eps < \eps_0} A_{k, \eps,M}) =0.
\end{align}
These claims immediately yield \eqref{max.radii} and \eqref{no.overlapping.borel} and conclude the proof of Lemma \ref{l.borel.cantelli}, $(i)$. 

\smallskip

The argument for \eqref{limsup.borel.1} and \eqref{limsup.borel.2} relies on an application of Borel-Cantelli's Lemma and is analogous to the one for \cite{GH}[Lemma 5.1].  We thus only sketch the proof. As shown in \cite{GH}[Proof of Lemma 5.1, (5.5) to (5.6)], up changing the constant $4$ in the definition of chain, we may reduce to prove \eqref{limsup.borel.1}-\eqref{limsup.borel.2} for a sequence $\{ \eps_j \}_{j \in \N}= \{ r^j\}_{j \in \N}$, with $r \in (0,1)$.

\smallskip

Using stationarity and the independence properties of the Poisson point process $(\Phi, \mathcal{R})$, it is easy to see that
\begin{align}\label{recursion.probability}
\P(A_{k, \eps, M})& \lesssim p_{0} p_1^{M-1}
\end{align}
where
\begin{align}
p_0&:= \P(\{\text{There is $z \in \Phi^\eps(D)$ with $\aeps \rho_z >\eps^{1- k \kappa }$}\}), \\
p_1&:= \P(\{\text{There is $z \in \Phi^\eps(B_{\eps^{1-(k+1)\kappa}(0)})$ with $\aeps \rho_z >\eps^{1- k \kappa }$}\}).
\end{align}
Using the moment condition \eqref{integrability.s} and provided $\kappa< \frac{\beta}{6}$ this yields
\begin{align}
p_0 \lesssim \eps^{\alpha\beta + (k \kappa- 1)(\frac 3 \alpha + \beta)}, \ \ \ p_1 \lesssim \eps^{\alpha \beta}.
\end{align}
Hence, by \eqref{recursion.probability}, we have that
\begin{align}
\P(A_{k, \eps,M})& \lesssim \eps^{\alpha\beta + (k \kappa- 1)(\frac 3 \alpha + \beta)}\eps^{(M-1)\alpha \beta}.
\end{align}  

\smallskip

On the one hand, if $\kappa < \min \bigl(\frac{\beta}{6} ;  \frac{\alpha^2\beta}{6+ 2\alpha\beta}\bigr)$, then we may pick $k_0:= \lfloor \frac{1}{\kappa} \frac{\alpha^2\beta}{6 + 2 \alpha \beta} \rfloor$ and observe that for every $k  \in \{ k_0, \cdots,  \km \}$ we have that
\begin{align}
\P(A_{k, \eps,1})& \lesssim \eps^{\frac 1 2 \alpha\beta}.
\end{align}
On the other hand, if $M \in \N$ is chosen big enough, for every $k=-3, \cdots, k_0$ also
\begin{align}
\P(A_{k, \eps,M}) \lesssim  \eps^{\frac 1 2 \alpha\beta}.
\end{align}   
Using these two bounds, we may apply Borel-Cantelli to the family of events $\{ A_{k, \eps_j, M}\}$ and conclude \eqref{limsup.borel.1} and \eqref{limsup.borel.2}.

\medskip

We now turn to case (ii). Identity \eqref{limsup.borel.2} may be rewritten as
\begin{align}
\P( \bigcup_{\eps_0> 0} \bigcap_{\eps < \eps_0} (\bigcup_{k =-3}^{k_0} A_{k, \eps,M})^c ) = 1.
\end{align}
This implies that for every $\delta > 0$, we may pick $\eps_0>0$ such that the set $\P(\bigcap_{\eps < \eps_0} (\bigcup_{k =-3}^{\km} A_{k, \eps,M})^c  )> 1-\delta$. The statement of $(ii)$ immediately follows if we set $A_\delta := \bigcap_{\eps < \eps_0} (\bigcup_{k =-3}^{\km} A_{k, \eps,M})^c$. The same argument  applied to \eqref{limsup.borel.1} implies the same statement for \eqref{max.radii}. 
\end{proof}

\begin{lem}\label{Kohn_Vogelius.general}
Let $\mathcal{Z}:= \{ z_i \}_{i \in I} \subset D$ be a collection of points and let $\mathcal{R}:=\{ r_i \}_{i\in I} \subset \R_+$ such that the balls $\{ B_{r_i}(z_i) \}_{i\in I}$ are disjoint. We define the measure 
\begin{align}\label{measure.M}
M:= \sum_{i \in I} g_i \delta_{\partial B_{r_i}(z_i)} \in H^{-1}(D),
\end{align}
where $g_i \in L^2( \partial B_{r_i}(z_i) )$.  Then, there exists a constant $C< +\infty$ such that for every Lipschitz and (essentially) disjoint covering $\{ K_j\}_{j \in J}$ of $D$ such that
\begin{align}\label{contains.balls}
B_{2r_i}(z_i) \subset K_j \ \ \ \text{OR} \ \ \ B_{r_i}(z_i) \cap K_j= \emptyset \ \ \ \ \text{for every $i \in I$, $j \in J$}
\end{align}
we have that
\begin{align}\label{KV.2}
\| M -m \|_{H^{-1}(D)} \leq C  \max_{j \in J}\mathop{diam}(K_j) \bigl( \sum_{i \in I} \|g_i \|_{L^2(\partial B_{r_i}(z_i))}^2 r_i^{-1}  \bigr)^{\frac 1 2},
\end{align}
with 
\begin{align}\label{mean.m}
m := \sum_{j \in J} \bigl(\frac{1}{|K_{j}|} \sum_{i \in I, \atop z_i \in K_j} \int_{\partial B_{r_i}(z_i)} g_i \bigr) \1_{K_{j}}.
\end{align}
\end{lem}

\begin{proof}{Proof of Lemma \ref{Kohn_Vogelius.general}}
This lemma is a simple generalization of \cite{G}[Lemma 5.1], where the harmonic functions $\{ \partial_n v_i \}_{i\in I}$ are replaced by a more general collection of functions $\{ g_i \}_{i\in I}$.
\end{proof}

\end{document}